\documentclass{amsart}
\usepackage{graphicx} 
\usepackage{epsfig}

\newtheorem{theorem}{Theorem}[section]
\newtheorem{lemma}[theorem]{Lemma}

\newtheorem{proposition}[theorem]{Proposition}
\newtheorem{corollary}[theorem]{Corollary}
\newtheorem{claim}[theorem]{Claim}

\newtheorem{conjecture}[theorem]{Conjecture}
\newtheorem{question}[theorem]{Question}

\theoremstyle{definition}

\newtheorem{example}[theorem]{Example}
\newtheorem{convention}[theorem]{Convention}
\newtheorem{remark}[theorem]{Remark}

\newtheorem*{thmSeiferter}{Theorem~\ref{thm:cisseiferter}}

\newtheorem*{thmLink0}{Theorem~\ref{link=0}}
\newtheorem*{propositionconstantHFK}{Proposition~\ref{prop:constantHFK}}
\newtheorem*{thmSeifertLspace}{Theorem~\ref{twist_seiferter_genus}}
\numberwithin{equation}{section}
\numberwithin{figure}{section}
\numberwithin{table}{section}

\newcommand{\comment}[1]{}
\newcommand{\lk}{\ell k}
\newcommand{\bdry}{\ensuremath{\partial}}

\DeclareMathOperator{\rk}{rk}

\DeclareMathOperator{\br}{br}
\newcommand{\nbhd}{\ensuremath{\mathcal{N}}}

\newcommand{\HFK}{\ensuremath{\widehat{\mathrm{HFK}}}}
\newcommand{\HF}{\ensuremath{\widehat{\mathrm{HF}}}}

\newcommand{\Q}{\ensuremath{\mathbb{Q}}}

\newcommand{\Z}{\ensuremath{\mathbb{Z}}}

\newcommand{\RP}{\ensuremath{\mathbb{RP}}}

\renewcommand{\)}{\textup{)}}

\begin{document}
\baselineskip 14pt

\title{Twist families of L-space knots, their genera, and Seifert surgeries}

\author{Kenneth L. Baker and Kimihiko Motegi}

\dedicatory{}

\begin{abstract}
Conjecturally, there are only finitely many Heegaard Floer L-space knots in $S^3$ of a given genus.  
We examine this conjecture for twist families of knots $\{K_n\}$ obtained by twisting a knot $K$ in $S^3$ along an unknot $c$ in terms of the linking number $\omega$ between $K$ and $c$. 
We  establish the conjecture in the case of $|\omega| \neq 1$, 
prove that $\{K_n\}$ contains at most three L-space knots if $\omega = 0$, 
and address the case where $|\omega| = 1$ under an additional hypothesis about Seifert surgeries. 
To that end, we characterize a twisting circle $c$ for which $\{ (K_n, r_n) \}$ contains at least ten Seifert surgeries. 
We also pose a few questions about the nature of twist families of L-space knots, 
their expressions as closures of positive (or negative) braids, and  their wrapping about the twisting circle.
\end{abstract}

\maketitle

{
\renewcommand{\thefootnote}{}
\footnotetext{
}
}
\section{Introduction}
\label{section:Introduction}
The Heegaard Floer homology $\HF(M)$ of a rational homology $3$--sphere $M$ satisfies $\rk \HF (M) \ge  | H_1(M; \mathbb{Z})|$.  
When this is actually an equality so that $\rk \HF(M) =  | H_1(M; \mathbb{Z})|$, then $M$ is an {\em L-space}.  
The set of L-spaces includes the lens spaces (except $S^1 \times S^2$) and all $3$--manifolds with finite fundamental group \cite[Proposition~2.3]{OS3} as well as many other Seifert fibered spaces \cite{OS3, LS}. 

A knot $K$ in the $3$--sphere $S^3$ is called an {\em L-space knot} if $K(r)$, 
the result of $r$--surgery on $K$, 
is an L-space for some $r \in \mathbb{Q}$. 
A non-trivial L-space knot is {\em positive} or {\em negative} according to the sign of $r$; 
only the unknot has both positive and negative L-space surgeries.  

Recall that the {\em knot Floer homology} of a knot $K \subset S^3$ is a  bi-graded, finitely generated abelian group $\HFK(K)$ that categorifies the Alexander polynomial $\Delta_K(t)$ \cite{OS_knot, Ras}, 
and that the knot Floer homology of an L-space knot has a particularly simple, constrained structure \cite{OS3}.

\medskip
This article takes motivation from a ``botany'' conjecture about the knot Floer homology of L-space knots of Hedden and Watson.

\begin{conjecture}[{\cite[Conjecture 6.7]{HW2}}]
\label{conj:finitehfk} 
Let $K$ be an L-space knot and with knot Floer homology $\HFK(K)$. 
Then there are only finitely many other knots whose knot Floer homology is isomorphic \(as bi-graded groups\) to $\HFK(K)$.
\end{conjecture}

We recast this as a conjecture about genera of L-space knots.
\begin{conjecture} 
\label{conj:genus}
Given an integer $N \geq 0$, 
there are only finitely many  L-space knots $K$ with $g(K) = N$. 
\end{conjecture}

\begin{proof}[Proof of equivalence of Conjectures~\ref{conj:finitehfk} and \ref{conj:genus}.]
Assume that Conjecture~\ref{conj:finitehfk} holds. 
Suppose for a contradiction that there are infinitely many L-space knots $K$ with $g(K) = N$ for some non-negative integer $N$. 
If necessary, by taking mirrors, 
we may assume that such L-space knots are positive. 
Since the degree of their Alexander polynomials is bounded above by $2g(K) = 2N$, 
and their non-zero coefficients are $\pm 1$ \cite[Corollary~1.3]{OS3}, 
there are only finitely many Alexander polynomials. 
Moreover, the Alexander polynomial of a positive L-space knot determines its $\HFK$ \cite[Theorem~1.2]{OS3}; 
see also \cite{Mano}. 
Thus infinitely many L-space knots share the same $\HFK$, contradicting the assumption. 

To prove the converse we assume Conjecture~\ref{conj:genus} holds, 
and suppose for a contradiction that there is an L-space knot $K$ for which there are infinitely many knots 
$K_i$ ($i = 1, 2, \dots$) with 
$\HFK(K_i) \cong \HFK(K)$ as bi-graded groups. 
By the rational surgery formula \cite{OS4}, 
any knot with $\HFK$ isomorphic to that of an L-space knot as bi-graded groups is also an L-space knot, 
and hence $K_i$ ($i = 1, 2, \dots$) is also an L-space knot. 
Also, since knot Floer homology detects genus \cite{OS_genus}, this also implies that $g(K) = g(K_1) = g(K_2) =\cdots$. 
This shows that infinitely many L-space knots have the same genus, contradicting the assumption. 
\end{proof}

\subsection{Twist families of knots}

In this article we examine Conjecture~\ref{conj:genus} for twist families of knots.  
The {\em twist family of knots} $\{K_n\}$ obtained by twisting a knot $K$ along a disjoint unknot $c$ is the sequence of knots that are the images of $K$ upon $(-\frac{1}{n})$--surgery on $c$ for $n \in \mathbb{Z}$. 
In the following we always assume that $c$ neither bounds a disk disjoint from $K$ nor is a meridian of $K$. 
Then it follows from \cite{KMS} that for each integer $m$, 
there are only finitely many integers $n$ such that $K_n$ is isotopic to $K_m$, 
in particular, the twist family $\{ K_n \}$ contains infinitely many distinct knots. 
Then Conjecture~\ref{conj:genus} for twist families of L-space knots is stated as: 

\medskip

\begin{conjecture} 
\label{conj:genus_twist}
For any twist family of knots $\{ K_n \}$ and any integer $N \ge 0$ 
there are only finitely many L-space knots $K_n$ such that $g(K_n) = N$. 
\end{conjecture}

Towards this conjecture, 
we first develop Theorem~\ref{genera_twist} which describes an asymptotic behavior of genera of knots under twisting in the general setting according to the linking number of $K$ and $c$.  
This theorem has the following direct consequence when the linking number is greater than $1$.

\begin{theorem}
\label{lk>1}
Let $\{ K_n \}$ be a twist family of knots obtained by twisting $K$ along $c$. 
If $|\lk(K,c)| > 1$, 
then $g(K_n) \to \infty$ as $|n| \to \infty$. 
In particular, Conjecture~\ref{conj:genus_twist} is true for any twist family of L-space knots with $|\lk(K,c)| > 1$. 
\end{theorem}

When the linking number is $0$, 
Theorem~\ref{link=0tightfiber} constrains the contact structures supported by the fibered knots and their mirrors in the twist family.  
The next theorem follows from this together with the fact that an L-space knot or its mirror is a fibered knot supporting the tight contact structure on $S^3$  \cite[Corollary~1.4 and Proposition~2.1]{Hed_positive}.

\begin{theorem}
\label{link=0}
Let $\{ K_n \}$ be a twist family of knots obtained by twisting $K$ along $c$.  
If $\lk(K,c) = 0$, 
then $K_n$ is an L-space knot for at most three integers $n$.
Furthermore, if $K_m$ and $K_n$ are L-space knots, then $|m - n| \le 2$.
\end{theorem}

In Theorem~\ref{link=0} we actually expect that there are at most two such integers $m, n$ with $|m -n| \le 1$. 
In contrast, 
for each integer $\omega > 1$ there are infinitely many twist families $\{ K_n \}$ 
each of which contains infinitely many L-space knots with $|\lk(K,c)| = \omega$; 
see \cite[Theorem~1.8]{Mote} and Subsection~\ref{twisted torus knots}.

\subsection{Twist families of surgeries}

Given a slope $r$ for $K$, 
then twisting along $c$ produces the twist family of knot-slope pairs $\{(K_n, r_n)\}$ 
called the {\em twist family of surgeries}, 
and the twist family of Dehn surgered manifolds $\{K_n(r_n)\}$.    
We call a knot-slope pair $(K, r)$ an {\em L-space surgery} if $K(r)$ is an L-space. 
Note that if $\omega = \lk(K,c)$, 
then $r_n = r_0+n \omega^2$.

\begin{remark}
\label{linear bound}
Given a twist family of surgeries $\{(K_n,r_n)\}$, 
there is a linear function of $n$ that bounds the genus of $K_n$ from above
whenever $K_n(r_n)$ is an L-space.  
This is due to the relation of genus and L-space surgery slope of Ozsv\'ath-Szab\'o \cite{OS4}.  
In particular, $g(K_n) \leq \tfrac{1}{2}(1+|r_0+n\omega^2|)$.  
\end{remark}

Let us specify Conjecture~\ref{conj:genus_twist} in terms of 
a twist family of surgeries. 

\medskip

\begin{conjecture}
\label{conj:genus_surgeries}
Let $\{ (K_n, r_n)\}$ be a twist family of surgeries.  
Then for any integer $N \ge 0$ there are only finitely many L-space surgeries $(K_n, r_n)$ such that 
$g(K_n) = N$. 
\end{conjecture}

Theorems~\ref{link=0} and \ref{lk>1} verify Conjecture~\ref{conj:genus_twist} 
and hence Conjecture~\ref{conj:genus_surgeries} for twist families obtained by twisting $K$ along $c$ when $|\lk(K,c)| \neq1$.   
In the case $\lk(K, c) = 1$, 
we prove:

\begin{proposition}
\label{prop:constantHFK}
Let $\{(K_n,r_n)\}$ be a twist family obtained by twisting $(K,r)$ along an unknot $c$ with $|\lk(K, c)| = 1$. 
If this family contains infinitely many L-space surgeries,  
then 
\begin{enumerate}
\item $\Delta_{K \cup c} (x,y) \doteq \Delta_{K}(x) \doteq \Delta_{K_n}(x)$ for all $n \in \Z$, 
\end{enumerate}
and there is an integer $N$ such that 
\begin{enumerate}
\item[(2)] $\HFK(K_n) \cong \HFK(K_N)$ for infinitely many integers $n$, and in particular
\item[(3)] $g(K_n) = g(K_N)$ for infinitely many integers $n$.  
\end{enumerate}
\end{proposition}

\begin{corollary}
\label{lk=1_finite}
Let $\{ (K_n, r_n)\}$ be a twist family of surgeries with $|\lk(K, c)| = 1$. 
If $g(K_n) \to \infty$ as $|n| \to \infty$, 
then $\{ (K_n, r_n)\}$ contains only finitely many L-space surgeries. 
\end{corollary}

\subsection{Twist families of Seifert fibered L-space surgeries}

Common examples of twist families of surgeries containing infinitely many L-space surgeries have infinitely many 
L-space surgeries in which the resulting manifolds are Seifert fibered; 
see \cite{Mote} for such examples.

\begin{convention}
\label{Seifert}
Throughout this article, 
we permit Seifert fibrations to have ``degenerate" fibers (i.e.\ index zero fibers). 
Accordingly, a Seifert fibered space is a $3$--manifold admitting a Seifert fibration with or without degenerate fibers. 
When we discuss surgeries, following the convention in \cite{DMM1}, 
we call a knot-slope pair $(K, r)$ a {\em Seifert surgery} if $K(r)$ is a Seifert fibered space in our generalized sense. 
See Section~2 in \cite{DMM1} for degenerate Seifert fibrations. 
Since connected sums of lens spaces are Seifert fibered L-spaces (in our sense), 
$(T_{p, q}, pq)$ is an L-space surgery and a Seifert surgery as well.
\end{convention}

The next result shows finiteness of L-space surgeries in a twist family $\{ (K_n, r_n) \}$ which contains at least $10$ Seifert surgeries:

\begin{theorem}
\label{twist_seiferter_genus}
Let $\{(K_n, r_n)\}$ be a twist family of surgeries obtained by twisting $(K, r)$ along an unknot $c$ with $|\lk(K, c)| = 1$. 
Assume that $(K_n, r_n)$ is a Seifert surgery for at least ten integers $n$. 
Then there are only finitely many L-space surgeries in the family.   
\end{theorem}

In the course of the proof of Theorem~\ref{twist_seiferter_genus}, 
we characterize a twist family of surgeries which contains a large number of Seifert surgeries. 
In doing so, we extend the foundational work of \cite{DMM1} on seiferters. 
See Section~\ref{sec:seiferter} for terminology and background regarding seiferters and pseudo-seiferters.  
Notably, we prove the following theorem.

\begin{theorem}
\label{thm:cisseiferter}
Let $\{(K_n,r_n)\}$ be a twist family of surgeries obtained by twisting $(K,r)$ along an unknot $c$ that is neither split from $K$ nor a meridian of $K$.  
If $(K_n, r_n)$ is a Seifert surgery for at least $10$ integers $n$, 
then $c$ is a seiferter or pseudo-seiferter for $(K, r)$.  
Consequently, $(K_n, r_n)$ is then a Seifert surgery for all integers $n$.
\end{theorem} 

Proposition~\ref{prop:pseudoseiferterlinking} gives the constraint that a pseudo-seiferter $c$ for a knot $K$ in $S^3$ must satisfy $|\lk(K,c)| \neq 1$.
However, we have not actually found any example of a pseudo-seiferter for a knot in $S^3$.

\begin{question}
Does there exist a pseudo-seiferter for a Seifert surgery on a knot in $S^3$?
\end{question}

\subsection{Notation and organization}

Throughout the paper we will use $N(*)$ to denote a tubular neighborhood of $*$ and 
use $\nbhd(*)$ to denote the interior of $N(*)$ for notational simplicity. 

\medskip

The rest of the paper is organized as follows. 
In Section~\ref{genera of twist knots} we investigate behavior of genera of knots under twisting 
operation using Alexander polynomials, 
and prove Theorem~\ref{genera_twist} which immediately implies Theorem~\ref{lk>1}. 
Proposition~\ref{prop:constantHFK} will be also proved in Section~\ref{genera of twist knots}. 
Section~\ref{sec:link=0} treats twist families of knots obtained by twisting $K$ along an unknot $c$ 
with $\lk(K, c) = 0$, 
and prove Theorem~\ref{link=0} as a corollary of the more general result 
Theorem~\ref{link=0tightfiber}. 
In Section~\ref{sec:seiferter} we will extend the foundational work of \cite{DMM1} on seiferters, 
and establish Theorem~\ref{thm:cisseiferter} and Proposition~\ref{prop:pseudoseiferterlinking}, 
which studies the linking of pseudo-seiferters. 
The proof of Theorem~\ref{twist_seiferter_genus} will be given in Section~\ref{sec:link=1}. 
In Section~\ref{braids}, 
we study Conjectures~\ref{conj:genus} and \ref{conj:genus_twist} from a viewpoint of braids, 
and provide examples of twist families of L-space knots whose twisting circles are not braid axes, 
but each L-space knot can be re-arranged as closures of positive or negative braids. 
Finally, in the last section, 
we will pose a few questions about the nature of twist families of L-space knots, 
their expressions as closures of positive (or negative) braids, and  their wrapping about the twisting circle.

\bigskip 

\section{Alexander polynomials and genera of knots in twist families}
\label{genera of twist knots} 

In this section we first prove the general result Theorem~\ref{genera_twist} below, 
which describes the behavior of  the genera of knots under the twisting operation. 
Then we will prove Proposition~\ref{prop:constantHFK} after preparing Lemma~\ref{lem:parity}. 

\medskip

\begin{theorem}
\label{genera_twist}
Let $\{K_n\}$ be the twist family of knots in a homology sphere obtained by twisting the knot $K$ along an unknot $c$.  
Then one of the following occurs:
\begin{enumerate}
\item $\lk(K,c) =0$ and $g(K_n)$ is constant for all but at most one $n$ for which $g(K_n)$ may be less,
\item $|\lk(K,c)| =1$ and $\Delta_{K \cup c} (x,y) \doteq \Delta_{K}(x)$, or
\item $|\lk(K,c)| \geq 1$ and $g(K_n) \to \infty$ as $|n| \to \infty$.
\end{enumerate}
\end{theorem}

Here $\Delta_L$ denotes the multivariable Alexander polynomial of the link $L$ and $\doteq$ signifies equivalence up to multiplication by a unit in the corresponding Laurent polynomial ring.  
For situations such as Theorem~\ref{genera_twist}(2), 
we regard  $\Z[x^{\pm1}]$ as the natural subring of $\Z[x^{\pm 1}, y^{\pm 1}]$. 

\medskip

\begin{question}
\label{ques:meridian}
Observe that if $c$ is a meridian of $K$ then $\Delta_{K \cup c} (x,y) \doteq \Delta_{K}(x)$ and $K_n = K$ for all $n$.   
If $|\lk(K,c)| =1$, $\Delta_{K \cup c} (x,y) \doteq \Delta_{K}(x)$, 
and $g(K_n) \leq N$ for some constant $N$, 
then must $c$ be a meridian of $K$?
\end{question}

\medskip

\begin{remark}
Conclusion (2) can occur even when $c$ is not a meridian of $K$.
Figure~\ref{fig:constantalex} shows a link $K \cup c$ with $|\lk(K,c)| =1$ such that $\Delta_{K \cup c}(x,y) = 3x^{-1}-5+3x$.  
Hence  $\Delta_{K_n}(t)=3t^{-1}-5+3t$ for all integers $n$.  
Is there an upper bound on the genera of these knots? 
\end{remark} 

\begin{figure}[h]
\begin{center}
\includegraphics[width=0.25\textwidth]{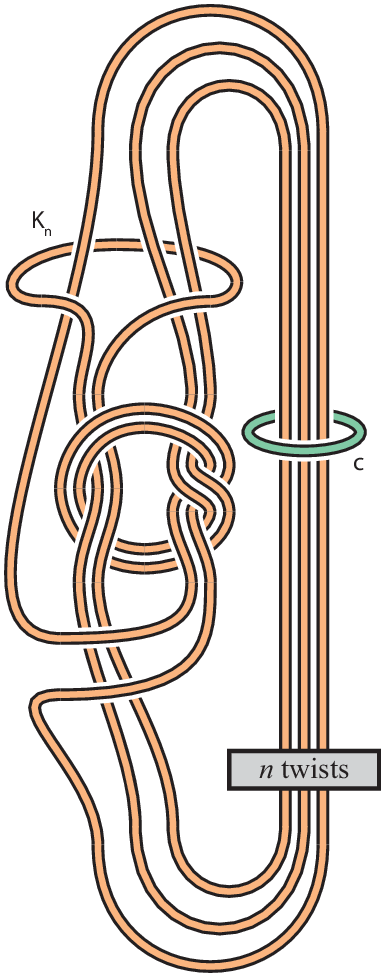}
\caption{The knots $K_n$ all have the same Alexander polynomial.}
\label{fig:constantalex}
\end{center}
\end{figure}

\medskip

\begin{remark}
Conclusion (3) with $|\lk(K,c)| =1$ does occur.   
For example, 
let us take the two-bridge link $B(18,7)$ \(which is $7_2^2$ in Rolfsen's table and $L7a5$ in Thistlethwaite's table \cite{Ro, knotatlas}\).  
Since both components are unknotted, one may choose either component to be $K$ and the other to be $c$. 
Then $|\lk(K,c)| =1$ and the multivariable Alexander polynomial is $\Delta(x,y) \doteq (x+y-1)(xy-x-y)$ \cite{linkinfo, knotatlas}.  
Since $\br_y\Delta(x, y) = 2$ (the $y$--breadth of $\Delta(x,y)$, defined below), 
the proof of Theorem~\ref{genera_twist} shows that $g(K_n) \to \infty$ as $|n| \to \infty$.  
\end{remark}

Before proving Theorem~\ref{genera_twist}, 
we prepare some notation. 
For a non-zero Laurent polynomial $p(t) \in  \Z[t^{\pm1}]$, 
its {\em breadth} $\br( p(t))$ is the difference between the minimum degree and maximum degree of $t$ in $p(t)$. 
For a non-zero Laurent polynomial $p(x,y) \in \Z[x^{\pm1},y^{\pm1}]$, 
its {\em $y$--breadth} $\br_y (p(x,y))$ is the difference between the minimum degree and maximum degree of $y$ in $p(x,y)$.  
We similarly define $\br_x(p(x,y))$.  
(The breadth, $y$--breadth, and $x$--breadth of the zero polynomial are defined to be $-\infty$.) 

Let $K$ be a knot in a homology sphere $M$ with Alexander polynomial $\Delta_K(t) \in \Z[t^{\pm1}]$ and Seifert genus $g(K)$.  
Then we have the inequality:
\[ \br (\Delta_K(t)) \leq 2g(K).\]

Let $L_1 \cup L_2$ be an oriented link in a homology sphere $M$ 
and $E$ the exterior $M - \nbhd(L_1 \cup L_2)$ of $L_1 \cup L_2$. 
The two variable Alexander polynomial of $L_1 \cup L_2$ is $\Delta_{L_1 \cup L_2} (x,y)$. 
With the Laurent polynomial ring $\Lambda = \Lambda[x^{\pm1},y^{\pm1}]$, 
this records the structure of $H_1(\widetilde{E})$ as a $\Lambda$ module with respect to the basis $ \langle [\mu_1], [\mu_2]\rangle $ of $H_1(E)$ where $\mu_i$ is an oriented meridian of $L_i$, $[\mu_1] \mapsto x$ and $[\mu_2] \mapsto y$, 
and the additive structure in $H_1(E)$ maps to the multiplicative structure in $\Lambda$  
(i.e.\  $a[\mu_1]+b[\mu_2] \mapsto x^a y^b$). 

Torres \cite{Torres} gives fundamental properties of the two-variable Alexander polynomial of an oriented link 
$L_1 \cup L_2$ with $\lk(L_1, L_2) =\omega$ and its relation to the Alexander polynomial of a component: 
\begin{align*}
\Delta_{L_1 \cup L_2}(x,y) & = x^m y^n \Delta_{L_1 \cup L_2}(x^{-1},y^{-1}) \mbox{ for some } m,n \in \Z, \tag{T1} \\
\Delta_{L_1 \cup L_2}( t, 1) &\doteq \frac{t^{\omega}-1}{t - 1} \Delta_{L_1} (t), \mbox{ and }  \tag{T2} \\
\Delta_{L_1 \cup L_2}(1,1) &= \pm \omega. \tag{T3}
\end{align*}

Note that reversing a component of an oriented link reverses the orientation of its meridian and hence inverts the corresponding variable in the Laurent polynomial ring.  
By the Torres Formula (T1), 
the Alexander polynomial of a two component link $L_1 \cup L_2$ is preserved up to equivalence upon reversing both components, 
so ostensibly a two component link has two inequivalent multivariable Alexander polynomials. 
In the following, as a matter of convenience, we choose orientations of $L_1$ and $L_2$ so that $\omega = \lk(L_1, L_2) \ge 0$.  
When $\omega=0$ we content ourselves with any choice of orientation. 

\begin{proof}[Proof of Theorem~\ref{genera_twist}]
We choose orientations of $K$ and $c$ so that $\lk(K, c) = \omega \ge 0$. 
(This choice  has no impact on the conclusions of the Theorem.) 
When $\omega=0$, 
the result follows from work of Gabai \cite[Corollary~2.4]{GabaiII}.
Henceforth assume $\omega\geq 1$.

Let $E = M - \nbhd(K \cup c)$ denote the exterior of $K \cup c$ where $M$ is the homology sphere containing $K \cup c$. 
Then $H_1(E) = \langle [\mu_K], [\mu_c]\rangle \cong \mathbb{Z} \oplus \mathbb{Z}$ 
where $\mu_K$ and $\mu_c$ are oriented meridians of $K$ and $c$ respectively. 
Let $\lambda_c$ be the preferred (oriented) longitude of $c$.  
Observe that $[\lambda_c] = \omega [\mu_K]$ in $H_1(E)$.

Now consider the family of links $K_n \cup c_n$ with exterior $E_n$ obtained by $(-\frac{1}{n})$--surgery on $c$.   
Observe that $E_n \cong E$ where $\mu_{K_n} \mapsto \mu_{K}$ and $\mu_{c_n} \mapsto \mu_c - n \lambda_c$.   
Thus, using that $[\lambda_c] = \omega [\mu_K] = \omega [\mu_{K_n}]$ 
so that $[\mu_{c}] \mapsto - n \omega[\mu_K] +[\mu_{c_n}]$ in $H_1(E)$, 
we have

 \[\Delta_{K_n \cup c_n}(x_n, y_n) = \Delta_{K \cup c}(x_n, x_n^{n \omega} y_n).\]

Applying the Torres Formula (T2) and the preceding equation, we obtain:

\[
\frac{t^{\omega}-1}{t - 1} \Delta_{K_n}(t) \doteq \Delta_{K_n \cup c_n}(t,1) = \Delta_{K \cup c}(t,t^{n \omega}). \tag{$\star$}
\]

Since $\omega \geq 1$,  
we have 
\[ 2g(K_n)  \geq \br( \Delta_{K_n}(t)) =  \br( \Delta_{K \cup c}(t,t^{n \omega}))- (\omega -1). \tag{$\star\star$} 
\]
Thus the genus of $K_n$ will eventually increase with $|n|$ provided that we have $ \br_y( \Delta_{K \cup c} (x,y)) >0$. 

Since $c$ is the unknot, $\Delta_c(y) =1$.  
Therefore
\[ \Delta_{K \cup c}(1,y) \doteq \frac{y^\omega - 1}{y-1} \Delta_c(y) =  \frac{y^\omega - 1}{y-1}\]
and thus $\Delta_{K \cup c} (x,y)$ has positive $y$--breadth when $\omega \geq 2$.   
Hence conclusion (3) holds when $\omega \geq 2$.

If $\omega=1$, 
then $\Delta_{K \cup c}(1,y) \doteq \Delta_c(y)=1$, 
which implies that $\Delta_{K \cup c}(x,y) \neq 0$. 
However, if $\br_y (\Delta_{K \cup c}(x,y))= 0$, 
then $\Delta_{K \cup c}(x,y)$ is expressed as $f(x)y^k$ for some polynomial $f(x)$ and integer $k$, 
and hence $\Delta_{K \cup c}(x,y) \doteq f(x)$. 
Therefore 
($\star$) implies that $\Delta_{K \cup c}(x,y) \doteq \Delta_{K \cup c}(x,1) = \Delta_K(x)$
and moreover that $\Delta_{K_n}(x) = \Delta_{K_n \cup c_n}(x, 1) 
= \Delta_{K \cup c}(x, x^{nw}) 
\doteq \Delta_{K \cup c}(x, 1) = \Delta_K(x)$
for all $n \in \Z$.
Thus if $\omega=1$, 
then either  $\Delta_{K \cup c}(x,y) \doteq \Delta_K(x)\doteq \Delta_{K_n}(x)$ for all $n \in \Z$ and conclusion (2) holds or  $\br_y (\Delta_{K \cup c}(x,y)) >0$ and conclusion (3) holds.
\end{proof}

\begin{lemma} 
\label{lem:parity}
Let $L_1 \cup L_2$ be an oriented link with $\lk(L_1, L_2) = \omega > 0$. 
Then, working $\mod 2$, we have   
\[ \br_x(\Delta_{L_1 \cup L_2}(x,y)) \equiv_2 \br_y(\Delta_{L_1 \cup L_2}(x,y)) \equiv_2 \omega-1.\]
\end{lemma}

\begin{proof}
This is an application of the Torres Formulas.   
First observe that $\Delta_{L_1 \cup L_2}(x,y) \neq 0$ by (T3) because $\omega>0$.  
Hence the $y$--breadth of $\Delta_{L_1 \cup L_2}(x,y)$ is a non-negative integer. 
If $\br_y(\Delta_{L_1 \cup L_2}(x,y)) =n$, 
then by multiplying by powers of $x$ and $y$ we may write 
$\Delta_{L_1 \cup L_2}(x,y) = \sum_{i=0}^{n} a_i(x) y^i$, 
where $a_{0}(x) \neq 0$ and $a_n(x) \neq 0$ (and possibly $0=n$).    
Then by (T1) we have

\begin{align*}
\sum_{i=0}^{n} a_i(x) y^i &=\Delta_{L_1 \cup L_2}(x,y)\\
&=x^m y^n \Delta_{L_1 \cup L_2}(x^{-1},y^{-1})\\
 &= x^m y^n \sum_{i=0}^{n} a_i(x^{-1}) y^{-i}\\
&= x^m  \sum_{i=0}^{n} a_i(x^{-1}) y^{n-i}  \\
&= x^m  \sum_{i=0}^{n} a_{n-i}(x^{-1}) y^{i} 
\end{align*}
so that $a_i(x) = x^m a_{n-i}(x^{-1})$.  
Hence $a_i(1)=a_{n-i}(1)$, 
and therefore $\br_y (\Delta_{L_1 \cup L_2}(x,y)) \equiv_2 \br(\Delta_{L_1 \cup L_2} (1,y))$.
By (T2), $\br(\Delta_{L_1 \cup L_2}(1,y)) = \br(\Delta_{L_2}(y)) + \omega-1$.   
Since the breadth of the Alexander polynomial of a knot is always even,  
$\br(\Delta_{L_1 \cup L_2}(1,y)) \equiv_2 \omega-1$.  
Thus $\br_y (\Delta_{L_1 \cup L_2}(x,y)) \equiv_2 \omega-1$.   

A similar proof shows $\br_x (\Delta_{L_1 \cup L_2}(x,y)) \equiv_2 \omega-1$.
\end{proof}

\medskip

\begin{propositionconstantHFK}
Let $\{(K_n,r_n)\}$ be a twist family obtained by twisting $(K,r)$ along an unknot $c$ with $|\lk(K, c)| = 1$. 
If this family contains infinitely many L-space surgeries,  
then 
\begin{enumerate}
\item $\Delta_{K \cup c} (x,y) \doteq \Delta_{K}(x) \doteq \Delta_{K_n}(x)$ for all $n \in \Z$,
\item $\HFK(K_n) \cong \HFK(K_N)$ for infinitely many integers $n$, and in particular
\item $g(K_n) = g(K_N)$ for infinitely many integers $n$.  
\end{enumerate}
\end{propositionconstantHFK}

\medskip

\begin{remark}
Of course, as in Question~\ref{ques:meridian}, 
we know of no examples of twist families that satisfy all the hypotheses of Proposition~\ref{prop:constantHFK} for which $c$ is not a meridian of $K$.
\end{remark}

\begin{proof}

Note that the assertion of the proposition holds for $\{ (K_n, r_n) \}$ if and only if that holds for 
the family $\{ (K^*_{-n}, -r_n) \}$ obtained by taking mirrors. 
So we may assume that there is an integer $N>0$ such that $(K_n, r_n)$ is an L-space surgery for infinitely many $n \geq N$. 
In the following we choose orientations of $K$ and $c$ so that $\omega = \lk(K, c) = 1$. 
Then, since $r_n = r_0+n$, 
by increasing $N$ if necessary we may assume $r_n > 0$ so that 
$K_n$ is a positive L-space knot for infinitely many $n \geq N$.

Since $K_n$ is a positive L-space knot, 
then $r_n \geq 2g(K_n)-1$ \cite{OS4}.
Then equation ($\star\star$) above (with $\omega = 1$) yields
\[ 
r_0+n \geq \br(\Delta_{K \cup c}(t,t^{n})) - 1.   
\]

Recall that, as in the proof of Lemma~\ref{lem:parity}, if $\br_y(\Delta_{K \cup c}(x,y))=\ell$, 
then we may write $\Delta_{K \cup c}(x,y)=\sum_{i=0}^{\ell} a_i(x) y^i$ 
where the $a_i(x)$ are polynomials such that $a_0(x) \neq 0$, $a_\ell(x) \neq 0$, 
and $x^k a_i(x^{-1}) = a_{\ell-i}(x)$  for all $i$ for some integer $k$.  

If $\ell = 0$, 
then $\br(\Delta_{K \cup c}(t,t^{n})) = \br a_0(t)$, 
which is constant, the difference between $\deg a_0(t)$ and the smallest exponent of $t$ occurring in $ \deg a_0(t)$. 
If $\ell > 0$,  
then $\Delta_{K \cup c}(t, t^{n}) 
= a_0(t) + \cdots + a_{\ell}(t)t^{n\ell}$
and, for sufficiently large $n (\ge N)$,  
$\br(\Delta_{K \cup c}(t,t^{n}))  = n \ell + C$ 
where $C$ is the difference between $\deg a_\ell(t)$ and the smallest exponent of $t$ occurring in $ \deg a_0(t)$.  
So the formula works for $\ell \geq 0$. 
Thus the inequality above becomes

\[ n(1-\ell) \geq C-1-r_0.\]

For this inequality to be true for sufficiently large $n \ge N$,
we must have $ 1-\ell \geq 0$, i.e.\ $0 \leq \ell \leq 1$. 
In particular, 
since $\omega = 1$ and $\ell$ do not have the same parity by Lemma~\ref{lem:parity}, 
$\ell=\br_y(\Delta_{K \cup c}(x,y))=0$. 
This implies $\Delta_{K \cup c} (x,y) \doteq \Delta_{K}(x)$ and thus, 
as in the proof of Theorem~\ref{genera_twist}, 
$\Delta_{K}(x) \doteq \Delta_{K_n}(x)$ for all $n \in \Z$, giving (1).

Since Alexander polynomials of positive L-space knots determine their $\HFK$, 
(2) now follows from (1) and the hypothesis that the twist family contains infinitely many L-space knots.
Since knot Floer homology detects genus \cite[Theorem~1.2]{OS_genus}, (3) follows from (2). 
\end{proof}

\bigskip

\section{L-space knots in twist families with linking number zero}
\label{sec:link=0}

As shown in Theorem~\ref{genera_twist},  
twisting $K$ along an unknotted circle $c$ with $\lk(K, c) = 0$, 
we obtain an infinite family of knots of bounded genus. 
If this family contains infinitely many L-space knots, 
Conjecture \ref{conj:genus} turns out to be not true. 
However Theorem~\ref{link=0} below, 
which follows from Theorem~\ref{link=0tightfiber} and the fact that an L-space knot or its mirror is a tight fibered knot \cite[Corollary~1.4 and Proposition~2.1]{Hed_positive},  
excludes this possibility.
(For convenience, we say a fibered knot whose associated open book decomposition supports the positive tight contact structure on $S^3$ is a {\em tight fibered knot}.)

\medskip

\begin{thmLink0}
Let $\{ K_n \}$ be a twist family of knots obtained by twisting $K$ along $c$.  
If $\lk(K,c) = 0$, 
then $K_n$ is an L-space knot for at most three integers $n$. 
Furthermore, if $K_m$ and $K_n$ are L-space knots, then $|m - n| \le 2$.
\end{thmLink0}

\begin{proof}
By Ni \cite{Ni, Ni2} (cf.\ \cite{Ghi, Juh}), 
if $K$ is an L-space knot, then $K$ is a fibered knot.  
If $K$ is an L-space knot with a positive L-space surgery, 
then $g(K)=\tau(K)$ \cite{OS3} (see also \cite[Corollary~1.4]{Hed_positive}) and the open book decomposition associated to $K$ 
supports the (positive) tight contact structure on $S^3$  \cite[Proposition~2.1]{Hed_positive}.  
That is, $K$ is a tight fibered knot.
Similarly, if $K$ is an L-space knot with a negative L-space surgery, 
then the mirror of $K$ is a tight fibered knot. 
The result now follows from Theorem~\ref{link=0tightfiber} below.
\end{proof}

\medskip

\begin{theorem}
\label{link=0tightfiber}
Let $\{ K_n \}$ be a twist family of knots obtained by twisting $K$ along $c$.  
If $\lk(K,c) = 0$, 
then $K_n$ or its mirror is a tight fibered knot for at most three integers $n$. 
Furthermore, if $K_m$ and $K_n$ are two such knots, then $|m - n| \le 2$.
\end{theorem}

\begin{proof}
If for any integer $n$, 
neither $K_n$ nor its mirror is a tight fibered knot, 
then there is nothing to prove.  
So we may assume,  if necessary by a reparametrization, 
that $K = K_0$ and either $K$ or its mirror is a tight fibered knot.  
It follows from \cite[Corollary~2.4]{GabaiII} that $K$ has a Seifert surface 
$F \subset E(K) = S^3 - \nbhd(K)$ which is disjoint from $c$ so that 
$g(K_n) \le g(F)$ with equality for all but at most one integer $n$, say $n_0$.  
(Cf.\ Theorem~\ref{genera_twist}(1).)
In particular, 
the image of $F$ under $(-\frac{1}{n})$--surgery on $c$ gives a minimal genus Seifert surface for $K_n$ in those cases of equality.

\medskip

\noindent
\textit{Case I}.  
$g(K) = g(F)$, i.e.\ $F$ is a fiber surface of $K$. 

Since $F$ is a fiber surface, by cutting the exterior $E(K)$ along $F$ one obtains a product manifold $F \times [0, 1]$.
Assume that $K_n$ (and hence its mirror) is also a fibered knot for some  integer $n \ne 0, n_0$.
Thus $K_n$ is a fibered knot with $g(K_n) = g(F)$, 
and since a fiber surface for a fibered knot is unique up to isotopy (e.g.\ \cite[Lemma 5.1]{EL} or \cite{Thu}),  
$F$ becomes a fiber surface $F_n$ of $K_n$ after $(-\frac{1}{n})$--surgery on $c$. 
Hence $(-\frac{1}{n})$--surgery on $c$ takes the exterior of $K \cup F$ to the exterior of $K_n \cup F_n$; 
i.e.\ this is a cosmetic surgery of $F \times [0, 1]$ such that  $ F \times \bdry[0,1]$ is preserved.  
Then Ni \cite[Theorem~1.1]{Ni3} shows that $c$ may be isotoped so that in the projection $\pi : F \times [0, 1] \to F$, 
either (i) the projection of $c$ has no crossings,  
or (ii) the projection of $c$ has just one crossing. 

The immersed annulus $\pi^{-1}(\pi(c))$ intersects $\bdry N(c)$ in two longitudes and two meridians for each crossing of $\pi(c)$. 
The slope of these longitudes is referred to as the {\em blackboard framing}.

\medskip

Assume first that the situation (i) happens. 
Let us assume that $K_n$ is a fibered knot as above for two integers $n = n_1, n_2$ other than $0$ and $n_0$. 
Then we have: 

\medskip

\begin{lemma}
\label{framing}
The blackboard framing is the preferred longitude of $c$.
\end{lemma}

\begin{proof}
Let $\gamma$ be the blackboard framing of $c$. 
Then $\gamma = x \mu + \lambda$ for some integer $x$, 
where $(\mu, \lambda)$ is a preferred meridian longitude pair of $c$ in $S^3$. 
By \cite[Theorem~1.1]{Ni3} the distance between the surgery slope $-\frac{1}{n}$ and $\gamma$ is one. 
Thus $| 1 + nx | = 1$ for the nonzero integers $n = n_1, n_2$. 
This then implies $x = 0$, i.e.\ $\gamma = \lambda$. 
\end{proof}

\medskip

Now isotope $c$ into the fiber surface $F$ for $K$; 
we continue to use the same symbol $c$ to denote the isotoped one. 
Then $c$ is essential in $F$, 
for otherwise, $c$ bounds a disk disjoint from $K$, contradicting the assumption. 
Since $c \subset F$ is unknotted in $S^3$ and its framing by $F$
is its preferred longitude (Lemma~\ref{framing}), 
$c$ is a ``twisting loop" as in \cite[Definition~2.1]{Y}. 

An essential loop $c$ in a surface $F$ is called {\em isolating} if it is the only boundary component of a connected subsurface of $F$. 
We note that  \cite[Theorem~1.1(2)]{Y} is missing the hypothesis that the twisting loop is {\em non-isolating} which is necessary for its proof.   
To apply this theorem, 
we need the following lemma.

\medskip

\begin{lemma}
\label{lem:isolatingloop}
If there is a twisting loop in an embedded surface $F\subset S^3$, then there is a twisting loop in $F$ that is non-isolating.
\end{lemma}

\begin{proof}
Assume there is an isolating twisting loop in a surface $F$. 
Among such loops, let $c$ be one that bounds the smallest genus of subsurfaces.  
Let $F_c$ be the subsurface bounded by $c$.  
Among disks that $c$ bounds, let $D$ be one that intersects $F_c$ transversally and minimally.  
Since the framing of $c$ by $F$ and $D$ agree, 
minimality ensures that $\mathrm{int}D \cap F_c = \emptyset$ in a neighborhood of $c$. 
If $\mathrm{int}D \cap F_c$ is not empty, 
then it consists of simple closed curves whose framings by $D$ and $F_c$ agree. 
Of these curves, let $c'$ be an innermost one in $D$. 
If $c'$ is parallel to $c$ in $F_c$, 
we can find another disk bounded by $c$ which intersects $F_c$ in fewer components, contradicting the assumption. 
By the minimality assumption (of genus of subsurfaces), 
$c'$ must be non-isolating in $F_c$ and hence in $F$.  
Since it bounds a subdisk of $D$, it is also a twisting loop.  

On the other hand, if $\mathrm{int}D \cap F_c$ is empty, 
then $D \cup F_c$ is a closed surface of positive genus which must compress in $S^3$.  
Let $D'$ be a compressing disk for $D \cup F_c$ which we may take to be disjoint from $D$.  
Then the framing of $\bdry D'$ by $D'$ and $F_c$ agree.  
Since $\bdry D'$ is an essential curve in $F_c$ and hence also in $F$, it is a twisting loop.  
Because $D'$ is a compressing disk for $D \cup F_c$, $\bdry D'$ is not parallel to $\bdry F_c$.
Therefore, by the minimality assumption (of genus of subsurfaces), 
$\bdry D'$ is non-isolating in $F_c$ and in $F$.
\end{proof}

\medskip

Then it follows from Lemma~\ref{lem:isolatingloop} and \cite[Theorem~1.1]{Y} that 
any contact structure supported by the open book with page $F$ will be overtwisted.   
Similarly, since the mirror of $F$ also contains a twisting loop, the mirror of $c$, 
any contact structure it supports will also be overtwisted. 
This contradicts our choice of $K = K_0$.
(Indeed, one may show that in the supported contact structures the non-isolating twisting loop can be isotoped to a Legendrian unknot that bounds an overtwisted disk.)  
Hence $K_n$ can be fibered for at most one integer $n_1$ ($\ne 0, n_0$), 
where $K_0$ or its mirror is a tight fibered knot. 
Therefore there are at most two non-zero integers $n_0$ and $n_1$ 
such that $K_0, K_{n_0}$ and $K_{n_1}$, 
or their mirrors are tight fibered knots in the family $\{ K_n \}$.

\begin{remark}
\label{Stallings twist}
Let us drop the condition of $K_0$ being a ``tight" fibered knot for the moment. 
Then our argument shows that if $K_n$ is a fibered knot for at least two integers $n_1, n_2$ \(other than $0, n_0$\), 
then $c$ is a curve in $F$ along which one may do a ``Stallings twist'' \cite{Sta}.  
It then further follows that for every member of the twist family $\{K_n\}$ the knot $K_n$ is fibered with fiber $F_n$ in which $c$ continues to be a twisting loop.  
\end{remark}

\medskip

Next assume that the situation (i) does not occur. 
Then we must have the situation (ii).   
Recall that $K=K_0$ is assumed to be a fibered knot.
First we observe that $F$ is incompressible in $E(K_n)$ for all integers $n$. 
It is sufficient to show that $F = F \times \{ 0 \}$ and $F = F \times \{ 1 \}$ remain incompressible in the resulting $3$--manifold $X_n$ obtained from $F \times [0, 1]$ after $(-\frac{1}{n})$--surgery on $c$ for all integers $n$. 
By symmetry, we show this only for $F = F \times \{ 0 \}$. 
Assume for a contradiction that $F = F\times \{0\}$ compresses in $X_n$ after $(-\frac{1}{n})$--surgery on $c$ for some $n$. 
Then \cite[Theorem~1.5]{Ni3} or \cite[Theorem~0.1]{ST} (see also \cite[Theorem~1.4]{Ni_fibred}) implies that the projection of $c$ has no crossings, 
contradicting the hypothesis of situation (ii). 
Next we show that there is at most one non-zero integer $n$ such that $K_n$ is also fibered. 
If $K_{n}$ is also a fibered knot for $n \neq 0$, 
then since $F$ is incompressible in $E(K_n)$ as observed above, 
$g(K_{n}) = g(F)$ and the fiber $F$ of $K$ becomes a fiber surface $F_n$ for $K_{n}$ after $(-\tfrac{1}{n})$--surgery on $c$ \cite[Lemma 5.1]{EL} (\cite{Thu}).
Thus $(-\frac{1}{n})$--surgery on $c$ is also a cosmetic surgery of $F \times [0, 1]$. 
Since the cosmetic surgery slope is exactly the blackboard framing \cite[Theorem~1.1]{Ni3}, 
this non-zero integer $n$ is unique.
Note that each projection gives the unique blackboard framing, 
i.e. a cosmetic surgery slope. 
Now we suppose that after an isotopy in $F \times [0, 1]$, 
$c$ may have another projection with exactly one crossing. 
Since its blackboard framing may be distinct from the previous one, 
 each slope is expressed as $x\mu + \lambda$ and $y\mu + \lambda$ for some  integers $x$ and $y$ using the preferred meridian longitude pair $(\mu, \lambda)$ of $c$ in $S^3$. 
Since the cosmetic surgery slope on $c$, 
which coincides with the blackboard framing,  
corresponds to a twisting, $x$ and $y$ must be $\pm 1$. 
Thus even if $c$ has multiple projections each of which has just one crossing,  
the blackboard framing is $+1$ or $-1$.
Therefore there are at most two integers $n_0, n_1$ with $\{ n_0, n_1 \} = \{ -1, 1 \}$ such that  
$K_0$, $K_{n_0}$ and $K_{n_1}$ are fibered knots in the family $\{ K_n \}$.  

\medskip

Finally let us prove that if $K_m$ and $K_n$ or their mirrors are tight fibered knots, 
then $|m - n| \le 2$. 
We reparametrize the family $\{ K_n \}$ so that $K = K_0$ or its mirror is a tight fibered knot as above and then take a closer look at the values $n_0$ and $n_1$.

First we assume the situation (i) happens; 
hence the projection of $c$ in $F \times I$ to $F$  has no crossings.  

Recall that $(-\frac{1}{n})$--surgery on $c$ compresses $F$ for at most one integer $n$ which we denote as $n_0$ should it exist;  
such a surgery is a $\bdry$--reducing surgery in $F \times I$. 
Since the blackboard framing of $c$ is the only slope of a Dehn surgery on $c$ in which $F$ compresses 
(see \cite[Theorem~1.5]{Ni3} or \cite[Theorem~0.1]{ST}), 
the slope $-\frac{1}{n_0}$ must be the blackboard framing.  
Hence if  $(-\frac{1}{n_0})$--surgery on $c$ compresses $F$, then $-\frac{1}{n_0}$ is an integer, and thus $n_0 = \pm 1$. 

Recall that $(-\frac{1}{n_1})$--surgery is a cosmetic surgery of $F \times [0, 1]$ and $n_1 (\ne 0)$ satisfies $|1+ n_1x| =1$ for some integer $x$. 
If $x = 0$, then the blackboard framing is the preferred longitude of $c$ (so there is no $n_0$ for which $(-\frac{1}{n_0})$--surgery on $c$ compresses $F$)
and  neither $K_0$ nor its mirror is a tight fibered knot; see the argument just before Remark~\ref{Stallings twist}.
This contradicts the assumption. 
So $x \ne 0$ and the equality implies $n_1 = \pm 1, \pm 2$. 
Summarizing, we see that if $K_n$ or its mirror is a tight fibered knot, then $n \in \{ -2, -1, 0, 1, 2\}$. 
If none of $K_{-2}$, $K_2$, or their mirrors are tight fibered knots, 
then $K_n$ or its mirror can be a tight fibered knot for at most three integers $n = -1, 0, 1$ providing the desired result.
Suppose that $K_2$ or its mirror is a tight fibered knot. 
Since $-\frac{1}{2}$ is not the slope of a $\partial$--reducing surgery, 
$g(K_2) = g(F)$ and we can replace $K_2$ with $K = K_0$ by reparametrization and apply the same argument to conclude that 
if $K_n$ or its mirror is a tight fibered knot, then $n \in \{ 0, 1, 2, 3, 4 \}$. 
Taking the previous restriction, we have only three integers $n = 0, 1, 2$ for which $K_n$ or its mirror can be a tight fibered knot. 
In the case where $K_{-2}$ or its mirror is a tight fibered knot, 
a similar argument shows that $K_n$ or its mirror can be a tight fibered knot for at most three integers $n = -2, -1, 0$. 
It follows that if $K_m$ and $K_n$ or their mirrors are tight fibered knots (without any reparametrization), then $|m - n| \le 2$. 

Suppose next that the situation (ii) does not happen, 
i.e. we have the situation (ii). 
Recall that $\{ n_0, n_1 \} = \{ -1, 1 \}$. 
Hence $K_n$ or its mirror can be a tight fibered knot for at most three integers $0, -1, 1$.  

\medskip

\noindent
\textit{Case II}.  $g(K) < g(F)$, i.e.\ 
$F$ is not a fiber surface of $K$ and the fiber surface of $K$ cannot be made disjoint from $c$. 
Then it turns out that $g(K_n) = g(F)$ for any integer $n \ne 0$ as we mentioned above. 
In this situation $n_0 = 0$ and $K (= K_0 = K_{n_0})$ or its mirror is a tight fibered knot. 
We may assume that $K_{n_1}$ or its mirror is also a tight fibered knot for some $n_1 \ne 0$, 
for otherwise $K$ is a unique knot in $\{ K_n \}$ such that $K$ itself or its mirror is tight fibered. 
Apply the argument in Case I to $K_{n_1}$ instead of $K = K_0$, 
we see that there are at most three knots including $K, K_{n_1}$ that themselves or their mirrors are tight fibered knots,
and if $K_m$ and $K_n$ are two such knots (without any reparametrization), then $|m - n| \le 2$. 
\end{proof}

\medskip

\begin{example}
\label{linking 0 twist}
Let $K \cup c$ be the Whitehead link depicted in Figure~\ref{link0twist}. 
Then the linking number between $K$ and $c$ is zero and 
the twist family $\{ K_n \}$ contains exactly two L-space knots $K = K_0$ and $K_1$.  
Even though $K_{-1}$ is also fibered, both $K_{-1}$ and its mirror support overtwisted contact structures.  
Hence $K_{-1}$ cannot be an L-space knot.
\end{example}

\begin{figure}[h]
\begin{center}
\includegraphics[width=0.25\textwidth]{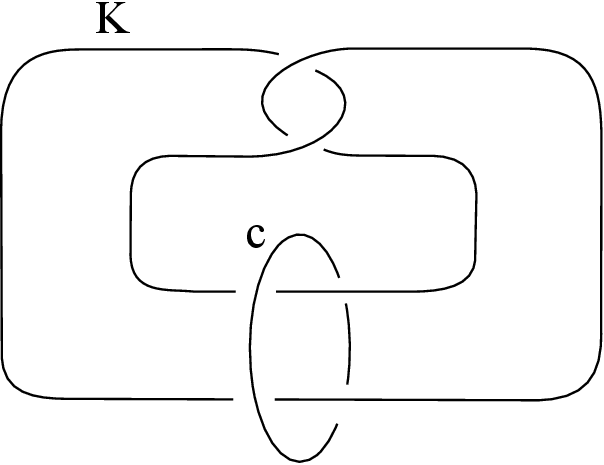}
\caption{The linking number between $K$ and $c$ is zero; $K = K_0$ is a trivial knot,  
$K_1$ is a trefoil knot, and $K_{-1}$ is the figure eight knot.}
\label{link0twist}
\end{center}
\end{figure}

\bigskip

\section{Twist families of Seifert surgeries; seiferters and pseudo-seiferters}
\label{sec:seiferter}

In this section we study when a twist family of surgeries may have a large number of Seifert surgeries without constraining the linking number $\lk(K, c)$. 
In doing so, we review and extend the foundations of \cite{DMM1}.  
Recall that the term {\em Seifert surgery} means a knot-slope pair $(K, r)$ in $S^3$ such that the result $K(r)$ of $r$--Dehn surgery on $K$ is a manifold that admits a Seifert fibration, possibly with degenerate fibers.  
If an unknot $c$ in the exterior of $K$ becomes isotopic to a fiber in a Seifert fibration of $K(r)$, 
then $c$ is called a {\em seiferter}:  
twisting the Seifert surgery $(K, r)$ along $c$ produces a $1$--parameter twist family $\{(K_n, r_n)\}$ of Seifert surgeries.   
Typically it is assumed that any disk bounded by $c$ is intersected by $K$ at least twice; 
otherwise $c$ is either split from $K$ or a meridian of $K$, sometimes called an ``irrelevant'' seiferter.  

Let $c$ be a seiferter for a Seifert surgery $(K, r)$.  
The exterior of $c$ is a solid torus $V=S^3-\nbhd(c)$ containing $K$ so that the manifold $V(K;r)$ resulting from $r$--Dehn surgery on $K$ in $V$ has a Seifert fibration.  
If $K(r)$ has a non-degenerate Seifert fibration, 
the main result of \cite{MM3} (see also \cite[Theorem 2.2]{DMM1}) 
shows that either $r \in \Z$ or $K$ is a torus knot in $V$ or a cable of a torus knot in $V$.
Hence the situation when $c$ is a seiferter for a (non-degenerate) Seifert surgery $(K, r)$ with $r \not \in \Z$ is well understood. 
Thus \cite{DMM1} focuses upon integral Seifert surgeries $(K, m)$ where $m\in \Z$.  
As we will observe in the proof of Lemma~\ref{integral}, 
even when $K(r)$ has a degenerate Seifert fibration, we see that $r \in \Z$ or $K$ is a torus knot. 
(In reference to notation for surgery slopes, we always take $m\in \Z$ while in general $r \in \Q$.)
Theorems~3.2 and 3.19 of \cite{DMM1} classify seiferters for integral Seifert surgeries $(K, m)$.

One generalization of a seiferter is that of a {\em pseudo-seiferter}, cf.\ \cite[Definition 8.4]{Mote}.  
Given a Seifert surgery $(K, r)$, 
an unknot $c$ in the exterior of $K$ is a pseudo-seiferter if $c$ is not a seiferter but $c$ is isotopic to the cable of a fiber in some Seifert fibration of $K(r)$ where the preferred longitude $\lambda$ of $c$ in $S^3$ becomes the cabling slope of $c$ in $K(r)$.  
In particular, the manifold $V(K; r)$ is a graph manifold that is the union along a torus of a Seifert fibered space $X$ and a cable space $W$; 
the slope $\lambda \subset \bdry V \subset \bdry W$ is the cabling slope of the cable space.

In the definition of a pseudo-seiferter, 
the condition that $\lambda$ becomes the cabling slope of $W$ is precisely what's needed for $W_n = W \cup _{-\frac{1}{n}} N(c)$, 
the filling corresponding to $(-\frac{1}{n})$--surgery on $c$, to be a solid torus.  
This allows the Seifert fibration of $X$ to extend to a Seifert fibration of $K_n(r_n)$.  
Hence again, twisting the Seifert surgery $(K, r)$ along $c$ produces a $1$--parameter family $\{(K_n, r_n)\}$ of Seifert surgeries.

\medskip

In the following two subsections we show that 
if a twist family of surgeries $\{(K_n,r_n)\}$ obtained from a surgery $(K, r)$ by twisting along an unknot $c$ 
contains ten Seifert surgeries, 
then
\begin{itemize}
\item (Theorem~\ref{thm:cisseiferter}) $c$ is either a seiferter or a pseudo-seiferter  and so each surgery $(K_n,r_n)$ is a Seifert surgery; 
and thence

\item (Proposition~\ref{prop:pseudoseiferterlinking}) there is no pseudo-seiferter $c$ for $(K, r)$ with $|\lk(K,c)| = 1$. 

\end{itemize}

\medskip

\begin{remark}
The Seifert fibrations in this article are permitted to have degenerate exceptional fibers.  
Do note, however, that a Seifert fibered space obtained by surgery on a knot in $S^3$ cannot have more than one degenerate fiber unless the knot is trivial and the surgery is the $0$--slope. 
See \cite[Proposition~2.8]{DMM1}.
\end{remark}

\subsection{Seifert surgeries in twist families}

Let $\{(K_n,r_n)\}$ be a twist family in $S^3$ obtained by twisting $(K,r)$ along an unknot $c$.  
Recall that $c$ neither bounds a disk disjoint from $K$ nor is a meridian of $K$. 
Let us write:  
\[\mathcal{S} = \{ n \in \mathbb{Z}\ |\ \textrm{$K_n(r_n)$ is a \(possibly degenerate\) Seifert fibered space} \}.\]
If $\mathcal{S} \ne \emptyset$, 
by reparametrization we assume $K(r)$ is a (possibly degenerate) Seifert fibered space. 

\medskip

The goal of this subsection is to prove Theorem~\ref{thm:cisseiferter}, 
though phrased slightly differently for its presentation here.  

\begin{thmSeiferter}
If $|\mathcal{S}| > 9$, 
then $c$ is either a seiferter or a pseudo-seiferter for $(K, r)$ and $\mathcal{S} = \Z$. 
\end{thmSeiferter}

\begin{proof}  
By the Inheritance Property \cite[Proposition~2.6]{DMM1}, 
$c$ is a seiferter (or a pseudo-seiferter) for $(K, r)$ if and only if 
$(K_n, r_n)$ is a Seifert surgery for which $c$ remains a seiferter (or a pseudo-seiferter) for any $n \in \mathbb{Z}$. 
So showing that  $c$ is a seiferter or pseudo-seiferter implies that $\mathcal{S} = \Z$.   
Hence we assume that $|\mathcal{S}| > 9$  and aim to show that $c$ is a seiferter or pseudo-seiferter.

Let $V = S^3 - \nbhd(c)$ be the solid torus exterior of $c$ which contains the knot $K$, 
and let $V(K; r) = K(r) - \nbhd(c)$.  
Use the preferred meridian-longitude slopes $\mu$ and $\lambda$ for $\bdry N(c)$ to parametrize slopes in both 
$\bdry V$ and $\bdry V(K; r)$.  
Then observe that $K_n(r_n)$ is the result of filling $V(K; r)$ along the slope $\mu - n \lambda$, 
i.e.\ $K_n(r_n) = V(K; r) \cup _{-\frac{1}{n}} N(c)$. 

Scharlemann's \cite{Sch} strengthening of Gabai's work on surgeries on knots in solid tori \cite{Gabai_solidtorus} shows that either 

\begin{enumerate}
\item $V(K; r)$ is a solid torus (and so either $K$ is a $0$--bridge braid in $V$ or $K$ is a $1$--bridge braid in $V$, 
see also \cite{Berge1});
\item $V(K; r) \cong W \# L(p,q)$, $K$ is a $(p,q)$--cable knot in $V$, $p \geq 2$, 
 $r$ is the cabling slope of $K$, and $W$ is some $3$--manifold with $\partial W = \partial V$; or
\item $V(K; r)$ is irreducible and $\bdry$--irreducible.
\end{enumerate}

Since $K_n(r_n) = V(K; r) \cup_{-\frac{1}{n}} N(c)$ is a Seifert fibered space for more than nine integers $n$, 
$V(K; r)$ is not hyperbolic \cite[Theorem~1.2]{LM}, cf.\ \cite{A}.\footnote{Indeed, 
Thurston's Hyperbolic Dehn Surgery Theorem \cite{T1, T2, BePe, PetPorti, BoileauPorti} implies that if $K_n(r_n)$ is not hyperbolic for infinitely many $n$, then $V(K;r)$ is not hyperbolic.  
However explicit bounds have been obtained on the number of non-hyperbolic fillings a hyperbolic manifold may have.  
While \cite{LM} determines the optimal bound for hyperbolic manifolds with one cusp, 
as suggested by \cite{A} it is conceivable fewer Seifert fibered fillings are needed for our particular situation.
Our argument also requires a bound for filling multiple cusps, 
in which case the distance between two non-hyperbolic filling is less than or equal to $8$; 
see \cite[Table~2.1]{Go_small}.}
Therefore $V(K;r)$ is either reducible, $\bdry$--reducible, Seifert fibered (with non-degenerate Seifert fibration), or toroidal.
If $V(K;r)$ is Seifert fibered, then $c$ is a seiferter; so let us assume $V(K;r)$ is not Seifert fibered.  
If $V(K;r)$ is $\bdry$--reducible but not reducible, then it is a solid torus and hence Seifert fibered.  
Thus we have two cases to consider: 
Either 
\begin{itemize}
\item[Case I:] $V(K;r)$ is reducible (as in (2) above), or
\item[Case II:]  $V(K;r)$ is toroidal, irreducible, $\bdry$--irreducible, and not Seifert fibered.
\end{itemize}

\medskip

\noindent
{\em Case I: $V(K;r)$ is reducible.}\\
If $V(K;r)$ is reducible, it has a lens space summand $L(p,q)$ with $p \geq 2$, $K$ is a cabled knot in $V$, 
and $r$ is the cabling slope.  
Say $K$ is a cable of a knot $J$ in $V$; 
$J$ is not a core of $V$ because $V(K; r)$ is not Seifert fibered. 
Hence $K_n$ is a cable of the knot $J_n$ obtained by twisting $J$ along $c$, and $r_n$ is the cabling slope.  
Since the unknot $c$ does not bound a disk that is either disjoint from $J$ or intersected by $J$ just once, 
$J_n$ becomes a trivial knot in $S^3$ for at most two integers $n$ \cite{Gabai_solidtorus} (cf. \cite{KMS, Ma}). 
In the following we take $n \in \mathcal{S}$ so that $J_n$ is not a trivial knot in $S^3$.  
So assuming $K_n(r_n)$ is Seifert fibered, either it is irreducible and thus just the lens space $L(p,q)$ 
or it is reducible and either $L(2,1) \# L(2,1)$ with no degenerate fibers or a connected sum of two lens spaces with one degenerate fiber (cf.\  \cite[Proposition~2.8]{DMM1}).  
For homological reasons, $K_n(r_n)$ cannot be $L(2,1) \# L(2,1)$.

Assume that $K_n(r_n)$ is a lens space for some $n \in \mathcal{S}$. 
Then we appeal to the classification of lens space surgeries on satellite knots \cite[Theorem 1]{BL}.  
Since $J_n$ is non-trivial in $S^3$, 
then it is a torus knot. Therefore $K_n$ is a cable of this torus knot in $S^3$ and $r_n$ is an integral slope intersecting the cabling slope once. 
Yet since a non-trivial knot cannot be expressed as a non-trivial cable of $J_n$ in more than one way,  
$r_n$ cannot also be a cabling slope.  
This is a contradiction.

Hence $K_n(r_n)$ is a connected sum of lens spaces for $n \in \mathcal{S}$. 
Greene showed that $K_n$ must be the cable of a torus knot where the surgery is along the cabling slope \cite{Greene}. 
Since we have chosen $n$ so that $J_n$ is nontrivial, 
this implies that $J_n$ is a nontrivial torus knot in $S^3$ for each $n \in \mathcal{S}$. 
Let us determine the position of $J$ in $V$. 

\medskip

\begin{claim}
\label{J cable}
$J$ is a $0$--bridge braid in $V$. 
In particular, $K$ is a cable of a $0$--bridge braid in $V$. 
\end{claim}

\begin{proof}
If $V - \nbhd(J)$ is Seifert fibered, 
then it is a cable space and we have the desired conclusion.  
So we exclude the remaining possibilities of $V - \nbhd(J)$ being hyperbolic, reducible, or toroidal. 
If hyperbolic, 
following \cite[Corollary~1.2]{GWuAnnular}, 
there are at most four integers $n$ such that $J_n$ is a nontrivial torus knot, 
a contradiction.   
If reducible, then $J$ must be contained in a ball in $V$; thus $c$ bounds a disk disjoint from $K$, a contradiction.  
Thus we assume that $V - \nbhd(J)$ is toroidal.   
Let  $\mathcal{T}$ be a family of tori that gives the torus decomposition of $V - \nbhd(J)$ in the sense of Jaco-Shalen \cite{JS} and Johannson \cite{Jo}
\footnote{We say that a family of tori $\mathcal{T}$ gives a torus decomposition of an irreducible $3$--manifold $M$, 
if each member of $\mathcal{T}$ is an essential torus and each decomposing piece (i.e.\ component) obtained by cutting $M$ along all of these tori is Seifert fibered or hyperbolic and no proper subfamily 
of $\mathcal{T}$ has this property.}.
See also \cite{Hat2}. 
Let $X$ be the decomposing piece which contains $\partial V$; $X \ne V - \nbhd(J)$. 
If $X \cup _{-\frac{1}{n}} N(c)$ is $\bdry$-irreducible for some $n \in \mathcal{S}$, 
then a component $T$ of $\partial (X \cup _{-\frac{1}{n}} N(c))$ is an essential torus in the torus knot space 
$S^3 - \nbhd(J_n) = (V - \nbhd(J)) \cup _{-\frac{1}{n}} N(c)$, a contradiction. 
Thus $X \cup _{-\frac{1}{n}} N(c)$ is $\bdry$-reducible for any $n \in \mathcal{S}$.  
Hence \cite[Theorem~2.0.1]{CGLS} shows that $X$ is a cable space and the distance between the slope $-\frac{1}{n}$ and that of the fiber slope of $X$ on $\partial V$ is at most one. 
This then implies that the fiber slope coincides with the longitudinal slope $\lambda$ of $c$ in $\bdry V$. 
Let $V_X \subset V$ be the solid torus bounded by $T = \bdry X - \bdry V$ so that $V = X \cup_T V_X$ and $V_X$ contains $J$ and $K$.
Then since $V$ is a solid torus, the meridian of $V_X$ must intersect a regular fiber of $X$ in $T$ just once.  
Therefore the core of $V_X$ is isotopic in $V$ to $\lambda$.  
In particular, there is a meridional disk of $V$ disjoint from $V_X$.  
Hence $c$ bounds a disk disjoint from $K$, 
a contradiction.
\end{proof}

\medskip

Thus $K_n$ is a cable of a torus knot $J_n$ and $c$ is a basic seiferter for the companion torus knot $J_n$. 
It follows from \cite[Proposition~8.7]{DMM1} that $c$ is a seiferter for $K_n(r_n)$, hence for $K(r)$. 

\bigskip

\noindent
{\em Case II: $V(K;r)$ is toroidal, irreducible, $\bdry$--irreducible, and not Seifert fibered.}\\
Since $V(K; r)$ is irreducible, 
except for at most two integers $n$, 
$K_n(r_n)$ is irreducible and
hence not a connected sum of lens spaces \cite[Theorem~1.2]{GLreducible}.  
Thus, 
if $K_n(r_n)$ is Seifert fibered, 
then it admits a non-degenerate Seifert fibration and it is a Seifert fibered space in the usual sense; 
see \cite[Proposition~2.8 (2)(3)]{DMM1}.

\begin{claim}
\label{claim:septori}
If there is an essential torus in $V(K;r)$, then it is separating.
\end{claim}

\begin{proof}
If not, then there exists a non-separating  torus $T \subset V(K;r) \subset K_n(r_n)$ for all integers $n$.  
Homological reasons then  imply that $r_n=0$ for all $n$.  

If $T$ compresses in $K_k(r_k)$ for some integer $k$, 
then the compression produces a non-separating $S^2$.  
Hence $K_k(r_k) = S^1 \times S^2$ and $K_k$ is the unknot \cite{GabaiIII}.  
By \cite[Theorem 4.2]{KMS}(\cite{Ma}), 
if there were another integer $k'$ for which $K_{k'}$ were an unknot, 
then $K \subset V$ must be homeomorphic to a $(2, 1)$--torus knot in a solid torus; 
in particular $V(K;r)$ would be a Seifert fibered space, contrary to assumptions.  
Thus there is at most one integer $k$ such that $T$ compresses.  
In particular, $T$ is essential in $K_k$ for all but at most one (which is at least eight) of the integers $k \in \mathcal{S}$. 

When $T$ is essential in $K_k(r_k)$ for some $k \in \mathcal{S}$, 
then since $K_k(r_k)$ is a Seifert fibered space over $S^2$ or $\RP^2$ the torus $T$ must be horizontal with respect to any Seifert fibration of $K_k(r_k)$, 
e.g.\ \cite[Proposition~1.11]{Hat2}.
Therefore there are at least $8$ integers $k$ for which $K_k(r_k) = V(K;r) \cup_{-\tfrac{1}{k}} N(c)$ must be a torus bundle over $S^1$ in which $T$ is a fiber. 
Indeed, since $T \subset V(K;r)$, we must have 
\[K_k(r_k)-\nbhd(T) = (V(K;r) - \nbhd(T)) \cup_{-\tfrac{1}{k}} N(c) \cong T \times[0,1]\] 
for these integers $k$.  

Since the free homotopy class of a knot in the product $T \times [0,1]$ is fixed, 
the free homotopy class of its projection to $T$ is also fixed.  
So, since $T$ is a torus, 
the homotopy class  of the projection may be represented as a multiple $m\mu$ of a primitive homology class $\mu$, 
and any projection of the knot with at most one crossing must lie in an annulus whose core curve represents $\mu$.  
Hence, if a knot does not have a projection with no crossings, then it has at most two non-isotopic projections with a single crossing. Applying \cite{Ni3}, such a knot could have at most $2$ non-trivial surgeries to a manifold homeomorphic to $T \times [0,1]$. 

Therefore, by \cite{Ni3},  
for each of these integers $k$ it must be that $c_k$ may be isotoped in $K_k(r_k)-\nbhd(T)$ into the boundary where the $0$--slope on $c_k$ agrees with the framing by the boundary torus; 
see the argument in the proof of Theorem~\ref{link=0}.   
(The slope of the blackboard framing of a one crossing projection whose single crossing is nugatory is accounted for by the associated slopes of its crossingless projection.) 
That is, $c_k$ is isotopic to a $0$--framed curve in the fiber $T$ of the torus bundle $K_k(r_k)$.  
Hence $K_n(r_n)$ must be a torus bundle for all integers $n$.   
Therefore for each integer $n$, the knot $K_n$ is a genus one fibered knot \cite[Corollary~8.23]{GabaiIII}, 
and so it is either a trefoil or the figure eight knot \cite{BuZ, GonAcu}. 
Since $c$ is not a meridian of $K$ and does not bound a disk disjoint from $K$, 
this contradicts that the same knot can only appear finitely many times in the twist family $\{K_n\}$ \cite[Theorem 3.2]{KMS}.
\end{proof}

Let $\mathcal{T}$ be a family of essential tori in $V(K; r)$ which gives a torus decomposition of $V(K; r)$.   
By assumption $\mathcal{T}$ is non-empty
and, as shown in Claim~\ref{claim:septori}, consists of separating tori.
Let $X$ be the decomposing piece which contains $\partial V$.

If $X$ is hyperbolic, 
then referring to Table~2.1 in \cite{Go_small} in which the relevant results from \cite{Go_Boundary, GLreducible, GL_TS, GW_AS, Oh, Q, Sch, Wu_incomp, Wu_suture} among others are summarized,  
we see that there are at most nine integers $k$ such that $X \cup_{-\frac{1}{k}}N(c)$ is not hyperbolic. 
Since $| \mathcal{S} | > 9$ we have an integer $n \in \mathcal{S}$ for which $X \cup_{-\frac{1}{n}}N(c)$ is also hyperbolic. 
But then $\mathcal{T}$ gives a torus decomposition for $K_n(r_n)$ 
in which we have the hyperbolic piece $X \cup_{-\frac{1}{n}}N(c)$, 
contradicting that $K_n(r_n)$ is Seifert fibered.

Hence $X$ admits a Seifert fibration.
Let $T$ be a component of $\partial X - \partial V$. 
We now divide into two cases depending on whether, in $X \cup_{-\frac{1}{n}} N(c)$, 
\begin{itemize}
\item[(a)] $T$ is compressible for at most two integers $n \in \mathcal{S}$, or 
\item[(b)] $T$ is compressible for more than two integers $n \in \mathcal{S}$. 
\end{itemize}
In the following we show that the first case does not occur and the second case leads us to conclude that 
$c$ is a pseudo-seiferter for $(K, r)$.

\medskip

{\em Case II\ }(a): Suppose that $T$ is compressible in $X \cup_{-\frac{1}{n}} N(c)$ for at most two integers $n \in \mathcal{S}$.  
We can choose $n \in \mathcal{S}$ so that $T$ is incompressible in $X \cup_{-\frac{1}{n}} N(c)$. 

Since $X$ admits a Seifert fibration, 
any Seifert fibration of $X$ extends to one on $X \cup_{-\frac{1}{n}} N(c)$. 
Note that since $T$ is incompressible in $X \cup_{-\frac{1}{n}}N(c)$, 
the extended Seifert fibration is non-degenerate \cite[Lemma~2.7]{DMM1} and unique except when $X \cup_{-\frac{1}{n}} N(c)$ is $S^1 \times S^1 \times [0, 1]$ or 
the twisted $I$--bundle over the Klein bottle \cite[VI.18]{Ja}.  
As Lemma~\ref{exceptional case} below shows, 
except for at most two integers $n$, 
these exceptional cases cannot occur. 
Hence we can take $n \in \mathcal{S}$ 
so that the Seifert fibration of $X \cup_{-\frac{1}{n}} N(c)$ is the extension of a Seifert fibration of $X$, 
which has a unique Seifert fibration.  
Furthermore, since $X \cup_{-\frac{1}{n}} N(c)$ is a (non-degenerate) Seifert fibered manifold with boundary, 
it is irreducible and hence $\bdry$--irreducible. 
Since $K_n(r_n)$ admits a Seifert fibration, 
the Seifert fibration now on $X$ must be compatible with that of the next decomposing pieces along the tori $\partial X - \partial V$. 
Thus $V(K; r)$ was already Seifert fibered, contradicting our original  assumption for Case II. 

\begin{lemma}
\label{exceptional case}
There are at most two integers $n \in \mathcal{S}$ such that 
$\hat{X}_n = X \cup_{-\frac{1}{n}} N(c)$ is $S^1 \times S^1 \times [0, 1]$ or 
a twisted $I$--bundle over the Klein bottle. 
\end{lemma}

\begin{proof}
For any $k \in \Z$ let $\hat{X}_k = X \cup_{-\frac{1}{k}} N(c)$ and denote the core of $N(c)$ in $\hat{X}_k$ as $c_k$.

For $n \in \mathcal{S}$, since $K_n(r_n)$ is Seifert fibered and $T$ is an incompressible, separating torus in $K_n(r_n)$, 
$T$ must be a horizontal or vertical torus in any Seifert fibration of $K_n(r_n)$, \cite[Proposition~1.11]{Hat2}.
If it is horizontal, 
then $T$ splits $K_n(r_n)$ into two twisted $I$--bundles over the Klein bottle, 
however Claim~\ref{claim:seifertoverklein} below shows this cannot occur.
Thus $T$ must be vertical, and the Seifert fibration restricts to the components of its complement.  
Moreover, at most one of these components has multiple Seifert fibrations (i.e.\ is a twisted $I$--bundle over the Klein bottle).

\begin{claim}
\label{claim:seifertoverklein}
A union of two twisted $I$--bundles over a Klein bottle cannot be
obtained by surgery on a knot in $S^3$.
\end{claim}

\begin{proof}
Let $M_i$ ($i = 1, 2$) be a twisted $I$--bundle over the Klein bottle, and 
let $M$ be a union of $M_1$ and $M_2$ in which $M_1 \cap M_2 = \partial M_1 = \partial M_2 = \Sigma$. 
Note that $\Sigma$ is the $\partial I$--subbundle of $M_i$, 
and $\pi_1(\Sigma)$ is an index two normal subgroup of $\pi_1(M_i)$ for $i = 1, 2$. 
Then by the van-Kampen theorem we have: 
\[1 \to \pi_1(\Sigma) \to \pi_1(M) \to \mathbb{Z}_2 \ast \mathbb{Z}_2 \to 1.\]
Therefore $H_1(M)$ has an epimorphism to $\mathbb{Z}_2 \oplus \mathbb{Z}_2$. 
Thus $M$ cannot be obtained by surgery on a knot in $S^3$.
\end{proof}

\medskip

\noindent{\em Case: Twisted $I$--bundle over the Klein bottle.}  

Assume that, 
for some $n \in \mathcal{S}$, 
$\hat{X}_n = X \cup_{-\frac{1}{n}} N(c)$ is a twisted $I$--bundle over the Klein bottle. 
Then $\hat{X}_n$ has exactly two Seifert fibrations: 
a Seifert fibration $\mathcal{F}_{D_n}$ over the disk with two exceptional fibers of indices $2$ and 
a Seifert fibration $\mathcal{F}_{M_n}$ over the M\"obius band with no exceptional fibers \cite[Lemma 1.1]{WW}. 
Observe that 
in $\bdry \hat{X}_n$ a regular fiber of $\mathcal{F}_{D_n}$ and a regular fiber of $\mathcal{F}_{M_n}$ intersect exactly once.

Since $X$ is a Seifert fibered space, 
any Seifert fibration of $X$ extends across $\hat{X}_n$ to one of these (non-degenerate) Seifert fibrations 
in which $N(c)$ is a fibered neighborhood of an exceptional or regular fiber. 
Accordingly $X$ is either

\begin{enumerate}
\item a Seifert fibered space over the annulus with one exceptional fiber of index $2$ if $c_n$ is an exceptional fiber of $\mathcal{F}_{D_n}$ 
\item a Seifert fibered space over the annulus with two exceptional fibers of indices $2$ if $c_n$ is a regular fiber of $\mathcal{F}_{D_n}$, or 
\item a circle bundle over  the once-punctured M\"obius band if $c_n$ is an exceptional fiber of $\mathcal{F}_{M_n}$.
\end{enumerate}

In each of these three cases we assume that the regular fiber has slope $\frac{x}{y}$ on $\partial V = \partial N(c)$ 
for some relatively prime integers $x,\ y$. 
Then the distance between the slope $\frac{x}{y}$ of the regular fiber and the slope $-\frac{1}{k}$ of the meridian of $N(c)$ in $\hat{X}_k$ is 
$|kx+y|$  for any integer $k$.
Therefore, if $\hat{X}_k$ is homeomorphic to a twisted $I$--bundle over the Klein bottle (and hence to $\hat{X}_n$) then the following must occur:
In the first case  $c_k$ must be an exceptional fiber of order $2$ and this distance $|kx+y|$ must be $2$; 
this is possible for at most two values of $k$.
In the second and third cases  $c_k$ must be a regular fiber and this distance $|kx+y|$ must be $1$; 
this is possible for at most two values of $k$ if $(x, y) \ne (0, \pm 1)$, 
and for infinitely many integers $k$ if $(x, y) = (0, \pm 1)$ in which case the regular fiber is the longitude of $c$. 
Hence  $\hat{X}_k$ is homeomorphic to a twisted $I$--bundle over the Klein bottle for at most two integers $k$ (including $k=n$) 
unless $c_n$ is a regular fiber of a Seifert fibration on $\hat{X}_n$ where the fibers meet $\bdry N(c)$ along the $0$--slope.  
Therefore we continue now assuming this latter exceptional situation.

Since $K_n(r_n)$ is Seifert fibered, 
one of the two Seifert fibrations on $\hat{X}_n$ is the restriction of a Seifert fibration on $K_n(r_n)$ and therefore 
compatible with the restriction of the Seifert fibration on  the complementary piece $Y = K_n(r_n) - \hat{X}_n = V(K;r) - X$.  
As noted above, $Y$ has a unique Seifert fibration. 
If this restricted Seifert fibration on $\hat{X}_n$ arises as the extension of a Seifert fibration on $X$, 
then $V(K;r)$ is Seifert fibered, contradicting assumption.  
Hence  on $\hat{X}_n$ the ``restricted'' Seifert fibration from $K_n(r_n)$ and the ``extended'' Seifert fibration from $X$ must be different; 
one is the Seifert fibration $\mathcal{F}_{D_n}$ over the disk and the other is the Seifert fibration $\mathcal{F}_{M_n}$ over the M\"obius band.

Let us first assume that $X$ has a Seifert fibration over the annulus so that the extended Seifert fibration of $\hat{X}_n$ is $\mathcal{F}_{D_n}$ and the restricted Seifert fibration is $\mathcal{F}_{M_n}$.  
Hence a regular fiber $t_{M_n}$ of $\mathcal{F}_{M_n} \cap T$ is a fiber of the Seifert fibration on $Y$.  
By assumption, $c_n$ is a regular fiber of $\mathcal{F}_{D_n}$ and there is a vertical annulus $A$ in the restriction to $X$ joining $\bdry N(c)$ to 
$T=\bdry \hat{X}_n$, meeting $\bdry N(c)$ along the $0$--slope.  
Thus $(-\frac{1}{n'})$--surgery on $c_n$  
(in terms of the slopes on $\bdry N(c)$) in $\hat{X}_n$ is realized by an annulus twist along $A$ in $X$. 
Since the fiber $t_{M_n}$ of $\mathcal{F}_{M_n} \cap T$ intersects $t_{D_n}$ once, 
a regular fiber $t_{M_{n'}}$ of $\mathcal{F}_{M_{n'}} \cap T$ intersects $t_{M_{n}}$ minimally $|n-n'|$ times.  
Also observe that $t_{D_n}$ is still a regular fiber of $\mathcal{F}_{D_{n'}} \cap T$, intersecting $t_{M_n}$ once.  
Hence if $n' \neq n$ then no fibration of $\hat{X}_{n'}$ is compatible with the Seifert fibration on $Y$.  
Thus $n$ is the only element of $\mathcal{S}$.

When $X$ has a Seifert fibration over the once-punctured M\"obius band 
so that the extended Seifert fibration of $\hat{X}_n$ is $\mathcal{F}_{M_n}$ and the restricted Seifert fibration is $\mathcal{F}_{D_n}$,
we apply the same argument to obtain the same conclusion. 

\medskip

\noindent
{\em Case:  $I$--bundle over the torus.}

Suppose next that $\hat{X}_n=X \cup_{-\frac{1}{n}} N(c)$ is $S^1 \times S^1 \times [0, 1]$ for some $n \in \mathcal{S}$. 
This can happen only when $\partial X - \partial V$ consists of two components $T_0$ and $T_1$; 
each $T_i$ bounds a $\partial$--irreducible $3$--manifold $Y_i$ in $V(K; r)$.     
Identify $T_i$ with $S^1 \times S^1 \times \{ i \}$ in $\hat{X}_n$ for $i=0,1$. 
As discussed above, since $K_n(r_n)$ is a Seifert fibered space, 
both $Y_0$ and $Y_1$ are Seifert fibered spaces and each either has a unique Seifert fibration or is a twisted $I$--bundle over the Klein bottle and therefore has exactly two Seifert fibrations.  
However, since $\hat{X}_n$ is a product torus, at most one  is a twisted $I$--bundle over the Klein bottle by Claim~\ref{claim:seifertoverklein}.

Assume that $\hat{X}_{n_j} \cong S^1 \times S^1 \times [0, 1]$ for three integers $n_j$ with $j = 0, 1, 2$. 
Then $c_{n_0} \subset \hat{X}_{n_0}$ admits non-trivial surgeries to $\hat{X}_{n_1}$ and $\hat{X}_{n_2}$.   
Since $X$ is Seifert fibered and $\hat{X}_{n_0} \cong S^1 \times S^1 \times [0, 1]$, 
$X$ Seifert fibers only as circle bundles over the annulus,
$c_{n_0}$ must be a regular fiber in one of these fibrations.  
Hence $c_{n_0}$ is the core of a vertical annulus in $\hat{X}_{n_0}$ joining $T_0$ and $T_1$, and 
 $X$ is a $2$--fold composing space, i.e.\ $\mbox{[disk with two holes]} \times\ S^1$. 
 This annulus restricts to two vertical annuli in $X$; for each $i=0,1$, let $A_i$ be the one connecting $\bdry N(c)$ and $T_i$.
Then $(-\frac{1}{n_j})$--surgery on $c_{n_0}$  (in terms of the slopes on $\bdry N(c)$) corresponds to some annulus twist along $A_1$. 

Since $K_{n_0}(r_{n_0}) = Y_0 \cup \hat{X}_{n_0} \cup Y_1$ is Seifert fibered,  
as discussed above, for any Seifert fibration $\mathcal{F}_{n_0}$ of $K_{n_0}(r_{n_0})$, the tori $T_0$ and $T_1$ are vertical (and parallel), 
and $\mathcal{F}_{n_0}$ restricts to Seifert fibrations on the components $\hat{X}_{n_0}$, $Y_0$, and $Y_1$.
By Claim~\ref{claim:seifertoverklein} 
at least one of $Y_0$ and $Y_1$, say $Y_0$, 
is not a twisted $I$--bundle over the Klein bottle. 
So we may assume $Y_0$ has a unique Seifert fibration, 
$\mathcal{F}_0$ that is the restriction of $\mathcal{F}_{n_0}$; 
and $Y_1$ has at most two Seifert fibrations, 
$\mathcal{F}_{1}$ that is the restriction of $\mathcal{F}_{n_0}$ and $\mathcal{F}_{1}'$ if the second Seifert fibration exists.  
Note that $\mathcal{F}_{1}$ and $\mathcal{F}_{1}'$ do not match on $\partial Y_1$. 
Let $A$ be a vertical annulus in the restriction of the fibration $\mathcal{F}_{n_0}$ to  $\hat{X}_{n_0} = S^1 \times S^1 \times [0, 1]$ 
such that $A \cap T_0$ is a regular fiber of $\mathcal{F}_0$ on $\partial Y_0$
and $A \cap T_1$ a regular fiber of $\mathcal{F}_{1}$ on $\partial Y_1$. 

Let us prove that if 
$K_{n_j}(r_{n_j}) 
= Y_0 \cup \hat{X}_{n_j} \cup Y_1 
= Y_0 \cup (S^1 \times S^1 \times [0, 1]) \cup Y_1$ is Seifert fibered for both $j = 1, 2$, 
then $A_1 \cap T_1$ is isotopic to $A \cap T_1$, which is a fiber of $\mathcal{F}_{1}$. 
Assume for a contradiction that $A_1 \cap T_1$ is not isotopic to $A \cap T_1$. 
Then $(-\frac{1}{n_j})$--surgery on $c_0^*$ is realized by an annulus twist $\phi_{n_j}$ along $A_1$, 
hence $\phi_{n_j}(A \cap T_1) \ne A \cap T_1$.   
Recall that in $K_{n_j}(r_{n_j})$,  
$\hat{X}_{n_j}$ and $Y_1$ are glued so that 
$\phi_{n_j}(A \cap T_1)$ is identified with the regular fiber of $\mathcal{F}_{1}$ on $\partial Y_1$. 
Hence the Seifert fibration of $X \cup_{-\frac{1}{n_j}} N(c) = S^1 \times S^1 \times [0, 1]$
(which coincides with $\mathcal{F}_{0}$ on $Y_0$) is not compatible with the Seifert fibration $\mathcal{F}_{1	}$ of $Y_1$. 
Thus $n_0$ is the unique integer such that 
$Y_0 \cup \hat{X}_{n_0} \cup Y_1 
= Y_0 \cup (S^1 \times S^1 \times [0, 1]) \cup Y_1$ is Seifert fibered for the Seifert fibration $\mathcal{F}_{1}$ of $Y_1$.   

If $Y_1$ has the second Seifert fibration $\mathcal{F}_{1}'$, 
$Y_0 \cup \hat{X}_{n_1} \cup Y_1 
= Y_0 \cup (S^1 \times S^1 \times [0, 1]) \cup Y_1$ may be Seifert fibered for the Seifert fibration $\mathcal{F}_{1}'$ of $Y_1$, 
but the above argument shows that $n_1$ is the unique such integer. 
Thus if $Y_0 \cup \hat{X}_{n_j} \cup Y_1 
= Y_0 \cup (S^1 \times S^1 \times [0, 1]) \cup Y_1$ is Seifert fibered for both $j = 1, 2$, 
then $A_1 \cap T_1$ is isotopic to $A \cap T_1$. 
This then implies that $c_0^*$ is isotopic to the regular fiber $A_1 \cap T_1$ and 
$V(K; r)$ is Seifert fibered contrary to assumption. 
Hence there are at most two integers $n$ such that $\hat{X}_{n} \cong S^1 \times S^1 \times [0, 1]$. 
\end{proof}

\medskip

{\em Case II\ }(b): Assume that $T$ is compressible for more than two integers $n \in \mathcal{S}$. 
Then $X$ is a cable space and the distance between the slope $-\frac{1}{n}$ and 
that of the fiber slope of $X$ is less than or equal to one \cite[Theorem~2.0.1]{CGLS}. 
Since we have at least three such integers $n$, 
the fiber slope of the cable space $X$ coincides with the preferred longitude of $c$. 
Hence $X \cup _{-\frac{1}{k}} N(c)$ is a solid torus for any integer $k$. 
Let $X'$ be the decomposing piece next to $X$; 
we will show that $V(K; r) = X' \cup X$ and $X'$ is Seifert fibered. 

First assume for a contradiction that we have yet another decomposing piece $X'' (\ne X, X')$ in $V(K; r)$. 
Again, since $X \cup _{-\frac{1}{n}} N(c)$ is a solid torus (with distinct meridional slopes for each integer $n$) 
and gives a Dehn filling of $X'$ for each $n \in \mathcal{S}$,  $X'$ cannot be hyperbolic (following the argument used for $X$).   
Hence we may assume $X'$ admits a Seifert fibration.
If some component of $\partial X' - T$ is compressible in $V(K; r) \cup_{-\frac{1}{n}} N(c)$ for more than two integers $n \in \mathcal{S}$,  
then \cite[Theorem~2.0.1]{CGLS} shows that $X \cup X'$ is a cable space, which is impossible 
because a cable space is atoroidal. 
So we may assume that  some component $T'$ of $\partial X' - T$ is incompressible in $X' \cup_T (X \cup _{-\frac{1}{n}} N(c))$ for all but at most two integers $n \in \mathcal{S}$.  
Applying the argument in (i) again implies that $V(K;r)$ is Seifert fibered giving us a contradiction. 

So $V(K; r)$ consists of two decomposing pieces $X$ and $X'$, where $X$ is a cable space. 
It remains to see that $X'$ is a Seifert fibered space. 
If $X'$ is not Seifert fibered, then since it is a decomposing piece of $V(K; r)$, 
it is hyperbolic. 
Note that $K_n(r_n) = X' \cup_T (X \cup_{-\frac{1}{n}} N(c))$, 
where $X \cup_{-\frac{1}{n}} N(c) = S^1 \times D^2$ for $n \in \mathcal{S}$. 
Thus $K_n(r_n) = X' \cup _{\gamma_n} (S^1 \times D^2)$ for some slope $\gamma_n$ on $\partial X'$ which varies 
depending on $n$. 
Then \cite[Theorem~1.2]{LM} shows that $K_n(r_n)$ is hyperbolic for some integer $n \in \mathcal{S}$, 
a contradiction. 
Hence $V(K; r) = X \cup X'$, 
where $X$ is a cable space (with $\partial X \supset \partial V$) and $X'$ is a Seifert fibered space.  

Recalling that the fiber slope of $X$ is the preferred longitude of $c$, 
$X \cup_{-\frac{1}{n}} N(c) = S^1 \times D^2$ in which $c^*$ is a cable of a core $t$ of this solid torus. 
In particular, $X \cup_{-\frac{1}{0}} N(c) = S^1 \times D^2$ in which $c$ is a cable of a core $t$ of this solid torus. 
Since $K(r) - \mathrm{int}(X \cup_{-\frac{1}{0}} N(c)) (= X')$ is Seifert fibered, 
$c$ is a pseudo-seiferter for $(K, r)$. 
\end{proof}

\medskip

\subsection{Pseudo-seiferters do not have linking number 1.}
\label{subsec:pseudo-seiferter link1}
It is known that, 
for each integer $\ell \ge 0$, there is a Seifert surgery $(K, r)$ which has a seiferter with $|\lk(K,c)| = \ell$ \cite{DMM1}. 
On the other hand, so far we have no example of a Seifert surgery with a pseudo-seiferter. 
In this subsection, we will prove that there is no Seifert surgery $(K, r)$ which has a pseudo-seiferter $c$ with $|\lk(K,c)| = 1$. 

\medskip

\begin{proposition}
\label{prop:pseudoseiferterlinking}
Assume $c$ is a pseudo-seiferter for a Seifert surgery $(K, r)$.  
Then $|\lk(K,c)| \neq 1$.  
\end{proposition}

To prove this, we first establish conventions and two lemmas.

Recall that if $c$ is a pseudo-seiferter for the Seifert surgery $(K,r)$, 
then $V(K;r) = W \cup X$ where $W$ is a cable space 
(a Seifert fibered space over the annulus with one exceptional fiber) and $X$ is a Seifert fibered space such that 

\begin{itemize}
\item $\bdry W = \bdry V \cup \bdry X$,
\item the exceptional fiber $\epsilon$ of $W$ has index $p\geq 2$.
\item $\lambda$, the preferred longitude of $c$, is the slope of a regular fiber of $W$ in $\bdry V$ (and hence is the cabling slope), 
\item    $\epsilon$ may be oriented so that $\lambda = p \epsilon$ in $H_1(W)$,  and
\item the manifolds $W_n = W \cup_{-\tfrac{1}{n}}N(c)$ are all solid tori with meridians $\mu_n$ in $\bdry X$.
\end{itemize}
In particular, we may regard $c$ as a torus knot with closed regular neighborhood $N(c)$ in the solid torus $W_0$.  
Then we view  $\bdry V$ as the boundary of the solid torus $N(c)$ 
with standard meridian-longitude basis $\mu, \lambda$ and $\bdry X$ as the boundary of the solid torus $W_0$ 
with standard meridian-longitude basis $\mu_0, \lambda_0$. 

\begin{lemma}
\label{lem:meridianoffilledcablespace}
 Let $\gamma$ be a regular fiber of $W$ in $\bdry X$, 
 oriented to be homologous to $\lambda$. 
 Then $\mu_0 \cdot \gamma = p$ and $\mu_n = \mu_0 - n p \gamma$.
\end{lemma}

\begin{proof}
Due to our choices of orientation, $p \mu = \mu_0$ in $H_1(W)$.  
Then because $\mu \cdot \lambda = 1$, 
we obtain that $\mu_0 \cdot \gamma = p$.  
Since $\mu - n \lambda$ in $\bdry V$ bounds a disk in $W_n$, 
we have $\mu_0 - n p \gamma = p \mu - n p \lambda = p(\mu - n \lambda) =0$ in $H_1(W_n)$.  
Since $W$ is a cable space with exceptional fiber of index $p$, 
the regular fiber $\gamma$ satisfies $\gamma = q \mu_0 + p\lambda_0$ in $H_1(\bdry X)$ for some integer $q$ coprime to $p$.  Therefore $\mu_0 - np \gamma$ can be written as $\mu_0 - np (q \mu_0 + p\lambda_0) = (1-npq)\mu_0 - p\lambda_0$.  
Since $1-npq$ and $p$ are relatively prime, this element is primitive in homology and represents a single essential curve. 
Thus the essential curve $\mu_0 - n p \gamma$ in $\bdry X$ must be the meridian $\mu_n$ of $W_n$.
\end{proof}

\medskip

\begin{lemma}
\label{lem:integralsurgery}
If $(K, r)$ is a Seifert surgery with a pseudo-seiferter $c$ such that $|\lk(K,c)| = 1$, 
then $r$ is an integer. 
\end{lemma}

\begin{proof}
Assume that $c$ is a pseudo-seiferter for $(K, r)$. 
Let $W_0$ be a fibered solid tubular neighborhood of a fiber $t$ in a (possibly degenerate) Seifert fibration $\mathcal{F}$ 
which contains $c$ in its interior as a cable of $t$:
$W_0 - \nbhd(c)$ is a cable space $W$ in which $t$ is an exceptional fiber of index greater than one. 
By definition the fiber slope of $W$ coincides with the longitudinal slope $\lambda$ on $\bdry N(c)$. 
(Hence $W_0(c; -\frac{1}{n}) = W \cup _{-\frac{1}{n}} N(c)$ is a solid torus for all integers $n$.)
It should be noted that the Seifert fibration of $W$ 
(in which the slope of the regular fiber agrees with the cabling slope $\gamma$) 
does not arise from the Seifert fibration $\mathcal{F} |_{W_0}$, 
because $c$ is not a fiber in $\mathcal{F}$. 
In particular, 
the slope of the fibration $\mathcal{F}|_{\bdry {W_0}}$ and the cabling slope $\gamma$ of $W$ in $\bdry W_0$ are distinct.

\medskip

First suppose that $\mathcal{F}$ is a degenerate Seifert fibration. 
Proposition~2.8 in \cite{DMM1} classifies degenerate Seifert fibrations of $K(r)$. 
(Proposition~2.8 in \cite{DMM1} treats the case when $r$ is an integer, 
but its proof works even when $r$ is rational.)
Then by  \cite[Proposition~2.8 (1)]{DMM1} it contains at most two degenerate fibers, 
and if there are two degenerate fibers, then $K(r) \cong S^1 \times S^2$ and $r = 0$.  
Let us assume that $\mathcal{F}$ has exactly one degenerate fiber $t$. 
Then  \cite[Proposition~2.8 (1)(ii)]{DMM1} shows that $K(r) - \nbhd(t)$ is a Seifert fibered space 
over the disk or the M\"obius band, 
and in the latter case $K$ is a trivial knot and $r = 0$. 
So we may assume that $K(r) - \nbhd(t)$ is a Seifert fibered space over the disk. 
We divide into two cases depending upon if the base orbifold is a disk with at least two cone points or 
a disk with at most one cone point. 
In the former case, $K(r)$ is a connected sum of at least two lens spaces; 
in fact, \cite[Proposition~2.8 (3)]{DMM1} shows that it has two lens space summands. 
It then follows from \cite{GLu} that $r \in \mathbb{Z}$. 
Suppose that we have the latter case.  
Then obviously $K(r)$ is a lens space and the degenerate Seifert fibration $\mathcal{F}$ has the following form. 
Let us take a solid torus $H_1$ which has a a non-degenerate Seifert fibration, 
and attach a solid torus $H_2$ to $H_1$ so that the meridian of $H_2$ is identified with a regular fiber of  $H_1$. 
Then the result is the lens space $K(r)$. 
Extend the Seifert fibration of $H_1$ to $H_2$ along meridian disks except the core of $H_2$ to obtain $\mathcal{F}$.  
Then the core of $H_2$ is a degenerate fiber in $\mathcal{F}$.   
First let us assume that $c$ is a cable of an exceptional fiber or a cable of a degenerate fiber, 
namely that $c$ is a cable of the core of $H_1$ or of $H_2$ respectively.  
Then $c$ is isotopic into the Heegaard torus $\bdry H_1 = \bdry H_2$, 
and we can change the Seifert fibrations of $H_1$ and $H_2$ so that $c$ is a fiber.  
Hence $c$ is a seiferter for $(K, r)$ and not a pseudo-seiferter, a contradiction. 
Thus we may assume that $c$ is a cable of a regular fiber in the non-degenerate fibered solid torus $H_1$. 
Then we may write $H_1 - \nbhd(c) = W \cup X$ where $W$ is a cable space with an exceptional fiber of index $p \ge 2$ and 
$\partial N(c) \subset \partial W$, 
and $X$ is a cable space with an exceptional fiber of index $r \ge 2$ and $\partial H_1 \subset \partial X$. 
Let $\gamma_W$ be a regular fiber of $W$ in $T = \partial W \cap \partial X$, 
and let $\gamma_X$ be a regular fiber of $X$ in $T$.  
Note that $\gamma_W$ and $\gamma_X$ represent distinct slopes on $T$, 
i.e. $\gamma_W \cdot \gamma_X \ne 0$. 
After $(-\frac{1}{n})$--surgery on $c$, 
which corresponds to an $n$--twist along $c$, 
$W_n = W \cup_{-\frac{1}{n}} N(c)$ is a solid torus with a meridian $\mu_n$. 
By Lemma~\ref{lem:meridianoffilledcablespace} 
$\mu_n = \mu_0 - n p \gamma_W$, 
and $\mu_n \cdot \gamma_X 
= (\mu_0 - n p \gamma_W) \cdot \gamma_X
= \mu_0 \cdot \gamma_X - np \gamma_W \cdot \gamma_X$.  
Since $p \ge 2$, $| \mu_0 \cdot \gamma_X| = 1$ and $\gamma_W \cdot \gamma_X \ne 0$,  
$|\mu_n \cdot \gamma_X| \ge 2$ if $|n| \ge 2$.  
Hence $H_1(c; -\frac{1}{n}) = W_n \cup X$ is a Seifert fibered space over the disk with two exceptional fibers 
if $|n| \ge 2$ (cf. \cite{Go_satellite}). 
Let us choose $n$ with $|n| \ge 2$. 
Then since the Seifert fibration of  $H_1(c; -\frac{1}{n}) = W_n \cup X$ is an extension of the Seifert fibration of $X$, 
a regular fiber on $\partial H_1(c; -\frac{1}{n})$ is a meridian of the solid torus $H_2$ in $K_n(r_n) = H_1(c; -\frac{1}{n}) \cup H_2$, 
and hence $K_n(r_n) = K_n(r_0 + n)$ is a connected sum of two lens spaces. 
This implies that $r_n = r_0 + n$ is an integer \cite{GLu}, 
and hence $r = r_0$ is an integer as claimed.

\medskip

In the following we now assume that $K(r)$ has a non-degenerate Seifert fibration $\mathcal{F}$ in which $c$ is a cable of some fiber. 
Performing $\lambda$--surgery on $c$, 
$S^3$ becomes $S^1 \times S^2$ since $c$ is an unknot, 
and $K(r)$ becomes a manifold with a lens space summand since $c$ is a cabled knot in $K(r)$ 
with cabling slope $\lambda$.  
We claim that this resulting manifold is actually reducible and not just this lens space.

First recall that, 
since $c$ is a pseudo-seiferter, 
$V(K;r) = X \cup W$ where $X$ is a Seifert fibered space and $W$ is the above cable space, 
and $K(r) = X \cup W_0$.
Furthermore $W_0(c; \lambda) \cong S^1 \times D^2 \# L(p,q)$ 
where $p >1$ is the index of the exceptional fiber in the cable space $W$ and $q$ is some integer coprime to $p$, 
and the meridian of the $S^1 \times D^2$ summand in $\bdry W_0(c; \lambda)$ coincides with the cabling slope $\gamma$ of $W$ in $\bdry W_0=\bdry X$.
Therefore,  $\lambda$--surgery on $c$ in $K(r)$ is the manifold 
\[(K \cup c)(r, \lambda) = X \cup W_0(c; \lambda) = X \cup (S^1 \times D^2 \# L(p,q)) = X(\gamma)  \# L(p,q).\]

Now assume that this manifold $(K \cup c)(r, \lambda)$ is irreducible, 
which implies that $X(\gamma) \cong S^3$.  
Since $X$ is a Seifert fibered space, 
$X$ must be the exterior of some torus knot, say $T_{a,b}$, 
and $\gamma$ is its meridian $\mu_X$ (still oriented to be homologous to $\lambda$ in $W$). 
Let $\lambda_X$ be its preferred longitude. 
Then since $K_n(r_n) = (K \cup c)(r, -\tfrac{1}{n}) = X \cup W_0(c;-\tfrac{1}{n})=X \cup W_n$ and 
$W_n=W_0(c; -\tfrac{1}{n})$ is a solid torus for every integer $n$, 
$K_n(r_n) = T_{a,b}(s_n)$ where $s_n$ is the slope of the meridian $\mu_n$ of 
the solid torus $W_n$ with respect to the basis $\mu_X, \lambda_X$ in $\bdry X = \bdry (S^3-\nbhd(T_{a,b}))$.  
If $\mu_0 = C\mu_X+D\lambda_X$, then since $\mu_X=\gamma$ and $\mu_0 \cdot \gamma = p$, 
we must have $D=-p$ and so $s_0= -\frac{C}{p}$.  
Since  $\mu_n = \mu_0 - n p \gamma$ by Lemma~\ref{lem:meridianoffilledcablespace}, 
then $\mu_n = (C-np)\mu_X + (-p)\lambda_X$ has slope 
$s_n = \frac{-C+pn}{p} = s_0+n$.  
Therefore, since $r_n = r_0 + n \omega^2 = r_0 + n$ where $\omega = |\lk(K, c)| =1$,  
$K_n(r_0 + n) = T_{a,b}(s_0+n)$ for all integers $n$. 
Then, writing $r_0 = \frac{A}{B}$ for some coprime integers $A, B$ with $B > 0$, 
because $|H_1(K_n(r_0 + n))| = |H_1(T_{a,b}(s_0+n))|$, 
we have $|A + Bn| = |-C+pn|$ for all integers $n$. 
This implies that $r_0 = \frac{A}{B} = -\frac{C}{p} = s_0$, 
and more generally $r_n = s_n$ for all integers $n$. 
Then by Ni-Zhang \cite[Theorem~1.3]{NZ} (\cite{McCoy}), 
$K_n = T_{a, b}$ for $n$ sufficiently positive or sufficiently negative.   
Hence it follows from \cite{KMS} that $c$ bounds a disk which intersects $K$ at most once, 
contradicting the assumption on the twisting circle $c$. 
This establishes that $(K \cup c)(r, \lambda)$ is reducible.

Recall that after $\lambda$--surgery on $c$, 
$S^3$ becomes $S^1 \times S^2$ which contains $K$.  
Let us observe that $S^1 \times S^2 - K$ is irreducible. 
If there is an essential $2$--sphere $S$ in $S^1 \times S^2 - K$, 
then by the primeness of $S^1 \times S^2$, 
it either  is non-separating or bounds a ball in $S^1 \times S^2$ containing $K$. 
If it were non-separating, 
then it would be non-separating in $S^1 \times S^2$ as well, 
and since $|\lk(K, c)| = 1$, 
$K$ generates $H_1(S^1 \times S^2)$ and this must intersect $S$ since the algebraic intersection number between $K$ and $S$ is one. 
This contradicts $S$ being a $2$--sphere in the exterior of $K$.  
If $S$ were to bound a ball in $S^1 \times S^2$ that contains $K$, 
then we would have $\lk(K,c) = 0$; 
but this is not the case. 

Thus viewing $K$ as a knot in $S^1 \times S^2 = V \cup_{\lambda} N(c)$ with irreducible exterior, 
$r$--surgery on $K$ produces a reducible manifold $V(K; r) \cup_{\lambda} N(c)$. 
Since $S^1 \times S^2$ is reducible as well,  
\cite[Theorem~1.2]{GLreducible} implies that $r \in \Z$.
\end{proof}

Let us turn to a proof of Proposition~\ref{prop:pseudoseiferterlinking}. 

\begin{proof}[Proof of Proposition~\ref{prop:pseudoseiferterlinking}.]
Assume for a contradiction that a Seifert surgery $(K, r)$ has a pseudo-seiferter $c$ with $|\lk(K,c)| = 1$. 
Then following Lemma~\ref{lem:integralsurgery} the surgery slope $r$ is an integer $m=m_0$.

Recall that, with $r=m$, 
$V(K;m) = W \cup X$ where $W$ is a cable space of order $p$ and $X$ is a Seifert fibered space 
(possibly with a degenerate Seifert fibration).
As in the proof of Lemma~\ref{lem:integralsurgery}, $K_n(m_n) = X(\mu_n)$ 
where $\mu_n$ is the meridian of the solid torus $W_n = W \cup_{-\tfrac{1}{n}} N(c)$.  
The curve $\gamma$ is a regular fiber of $W$ in $\bdry X$, 
and as in Lemma~\ref{lem:meridianoffilledcablespace}, $\mu_n = \mu_0 - n p \gamma$ in $H_1(\bdry X)$.

Let $C$ be the core of the solid torus $W_0$ in $K(m) = K_0(m_0) = X(\mu_0)$. 
Isotope $C$ so that $C \cap N(K^*) = \emptyset$, where $K^*$ is the surgery dual of $K$. 
Then $C \subset K(m) - \nbhd(K^*) = S^3 - \nbhd(K)$. 
Now we can view $C$ as a knot in $S^3$ disjoint from $K$ that becomes isotopic to the core of the solid torus $W_0$ after $m$--surgery on $K$.
(Furthermore, we may view $c$ as a cable of $C$  in $X(\mu_0)$.)  
Then $K_n(m_n) = X(\mu_n)$ may be presented as Dehn surgery on the link $K \cup C$ 
where $m$--surgery is done as before on $K$ and $\mu_n$--surgery is done on $C$.

Let us now use this to obtain an explicit computation of homology.
Let $\lambda_0$ be the preferred longitude of $C$ in $S^3$.  
Then $\gamma$, a regular fiber of $W$ in $T=\bdry X \subset \bdry W$,
may be expressed as the curve $p \lambda_0 + q \mu_0$ for some integer $q$ coprime to $p$, 
since $\mu_0 \cdot \gamma = p$ by Lemma~\ref{lem:meridianoffilledcablespace}. 
Hence $\mu_n = \mu_0 - np(p \lambda_0 + q \mu_0) = - np^2 \lambda_0 + (1-npq)\mu_0$.   
Thus $\mu_n$--surgery is $\tfrac{npq-1}{np^2}$--surgery in standard coordinates.  

To simplify exposition, 
we apply a ``slam-dunk" move \cite[p.163]{GS}. 
Let $C'$ be a meridian of $C$.  
Then $\mu_n$--surgery on $C$ can be viewed as $0$--surgery on $C$ and $\tfrac{-np^2}{npq-1}$--surgery on $C'$.   

Now we may obtain the presentation matrix $M_n$ given below for the homology of $K_n(m_n) = X(\mu_n)$ from its surgery presentation on the link $K \cup C \cup C'$:

\[
M_n = \left(
\begin{array}{ccc}
* &* & 0 \\
* & 0 & 1 \\
0 &  npq-1   & -np^2
\end{array}
\right).
\]

Since $K_n(m_n)=X(\mu_n)$ is a rational homology sphere obtained by integral $m_n$--surgery on the knot $K_n$ in $S^3$, 
\[|m_n| = |H_1(X(\mu_n))|  = |\det(M_n)| = |A(-np^2)-B(npq-1)| = |np(Ap+Bq)-B|\]
for some integers $A$ and $B$. 
Thus, because $m_n = m_0 + n \omega^2$ where $\omega = \lk(K,c)$, 
we have that for a sufficiently large integer $n$, 
\begin{align*}
\lk(K,c)^2 &= |H_1(X(\mu_{n+1}))| - |H_1(X(\mu_n))| \\
&=|(n+1)p(Ap+Bq)-B| - |np(Ap+Bq)-B|\\
&= |p(Ap+Bq)|. 
\end{align*}

This is impossible, because $p\geq 2$ and $|\lk(K,c)| = 1$.
\end{proof}

\bigskip

\section{L-space surgeries in twist families with linking number one}
\label{sec:link=1}

Our goal in this section is to prove: 

\medskip

\begin{thmSeifertLspace}
Let $\{(K_n, r_n)\}$ be a twist family of surgeries obtained by twisting $(K, r)$ along an unknot $c$ with $|\lk(K, c)| = 1$; 
$c$ is not a meridian of $K$. 
Assume that $(K_n, r_n)$ is a Seifert surgery for at least ten integers $n$. 
Then there are only finitely many L-space surgeries in the family.   
\end{thmSeifertLspace}

\noindent
\textit{Proof of Theorem~\ref{twist_seiferter_genus}.}
Suppose for a contradiction that $\{ (K_n, r_n) \}$ contains infinitely many L-space surgeries.  
Since $\{ (K_n, r_n) \}$ contains more than nine Seifert surgeries, 
then by Theorem~\ref{thm:cisseiferter} 
$(K, r)=(K_0,r_0)$ is a Seifert surgery and
$c$ is a seiferter or a pseudo-seiferter for $(K, r)$. 
Since $|\lk(K,c)| = 1$, 
by Proposition~\ref{prop:pseudoseiferterlinking} $c$ is not a pseudo-seiferter, 
and thus $c$ is a seiferter for $(K, r)$. 

\medskip

\begin{lemma}
\label{integral}
Let $(K, r)$ be a Seifert surgery which has a seiferter $c$ with $|\lk(K, c)| = 1$. 
Then $r$ is an integer.
\end{lemma}

\begin{proof}
Let $\mathcal{F}$ be the Seifert fibration in which $c$ is a (possibly degenerate) fiber.
In the following, we always consider this Seifert fibration for $K(r)$. 

First suppose that $\mathcal{F}$ is a non-degenerate Seifert fibration. 
Then $K(r) - \nbhd(c) = V(K, r)$ is a Seifert fibered space with non-degenerate Seifert fibration, 
and \cite{MM3} shows that either $r$ is an integer or $K$ is a $0$--bridge braid or a cable of a $0$--bridge braid in $V$. 
In the latter cases, $|\lk(K, c)| \ge 2$, contradicting the assumption. 
Thus $r$ is an integer. 

Now let us assume that $\mathcal{F}$ is a degenerate Seifert fibration.  
Proposition~2.8 in \cite{DMM1} classifies such Seifert fibrations of $K(r)$ for $r \in \Z$. 
Since the argument in its proof does not depend on the assumption that $r \in \Z$, 
we are able to conclude that 
$K(r)$ is a lens space or a connected sum of lens spaces. 
In the former either $r \in \Z$ or $K$ is a torus knot \cite{CGLS}, 
and in the latter $r \in \Z$ \cite{GLu}.
We proceed to exclude the possibility that $K$ is a torus knot and $K(r)$ is a lens space with a degenerate Seifert fibration under the assumption 
$|\lk(K,c)| = 1$. 
Proposition~2.8 in \cite{DMM1} further shows that if the Seifert fibration of $K(r)$ has more than one degenerate fiber, 
then $K$ is the trivial knot and $r=0$.  
Thus we may assume the Seifert fibration has just one degenerate fiber and $K(r) = H_1 \cup H_2$,  
where $H_1$ is a non-degenerate Seifert fibered solid torus such that its core is a fiber and $H_2$ is a degenerate fibered solid torus such that  its core is a degenerate fiber. 
If $c$ is either the core of $H_1$ or the degenerate fiber that is the core of $H_2$,  
then $K(r) - \nbhd(c) = V(K, r)$ is a solid torus, where $V = S^3 - \nbhd(c)$. 
Then following \cite{Gabai_solidtorus} $K$ is a $0$-- or $1$--bridge braid in $V$, 
and $|\lk(K, c)| \ge 2$, a contradiction. 
Hence $c$ is a regular fiber in the degenerate Seifert fibration of the lens space $K(r)$ and not isotopic to the core of $H_1$.
By an isotopy we may assume $c$ is a regular fiber in $H_1$. 
Note that the core of $H_1$ is an exceptional fiber of index $\ge 2$, 
for otherwise $c$ would be isotopic to the core of $H_1$. 
Then $K(r) - \nbhd(c) = V(K, r)$ is a connected sum of $S^1 \times D^2$ and a nontrivial lens space. 
It follows from \cite{Sch} that $K$ is cabled in $V$, 
and hence $|\lk(K, c)|$ cannot be one. 
\end{proof}

\medskip

Due to Lemma~\ref{integral},  
we now take $r$ to be an integer $m$ in what follows.

Since $c$ is a seiferter, 
 $(K_n, m_n)$ is a Seifert surgery for all $n$ by the Inheritance Property \cite[Proposition~2.6]{DMM1}. 
Recall that $m = m_0$ and $m_n = m_0 + n \lk(K,c)^2 = m_0 + n$. 

\medskip

\begin{lemma}
\label{link=1seiferter}
Let $c$ be a seiferter for a Seifert surgery $(K, m)$ with $|\lk(K,c)| = 1$. 
Then either $K$ is a torus knot and $c$ is a meridian of $K$,
or $c$ is a hyperbolic seiferter, i.e.\ $S^3 - K \cup c$ is hyperbolic. 
\end{lemma}

\begin{proof}
We apply the classification theorems of seiferters for Seifert surgeries, Theorems~3.2 and 3.19 in \cite{DMM1}. 
Suppose for a contradiction that we have a seiferter $c$ for $(K, m)$ with $|\lk(K,c)| = 1$ which is neither a meridian of $K$ nor a hyperbolic seiferter for $(K, m)$. 

Let us take a solid torus $V = S^3 - \nbhd(c)$ (with the core $C_V$), 
which contains $K$ in its interior.  
Among descriptions of $K \cup c$ in Theorems~3.2 and 3.19 in \cite{DMM1}, 
we divide them into three cases:  

\begin{enumerate}
\item 
$K \cup c$ is not of \textbf{types 3} or \textbf{4}  in \cite[Theorem~3.2]{DMM1}. 
Since we are excluding the case where $c$ is a hyperbolic seiferter, 
$K \cup c$ is of \textbf{type 1}. 
Since $c$ is not a meridian of $K$,  
it is an exceptional fiber in a Seifert fibration for $K = T_{p, q}$, 
i.e. $K$ is a $0$--bridge braid in $V$. 
Hence $|\lk(K,c) | \ge 2$, a contradiction.

\item 
$K \cup c$ is of \textbf{type 3} in \cite[Theorem~3.2]{DMM1}. 
Then there is a knotted solid torus $V'$ disjoint from $c$ which contains $K$ in its interior and 
whose core $C_{V'}$ is a $0$--bridge braid in $V = S^3 - \nbhd(c)$; $C_{V'}$ is a nontrivial torus knot in $S^3$. 
Thus $|\lk(C_{V'}, c)| \ge 2$, and $|\lk(K,c) |= |x| |\lk(C_{V'}, c)|$ cannot be $1$, 
where $[K] = x [C_{V'}] \in H_1(V'; \mathbb{Z})$. 

\item
$K \cup c$ is of \textbf{type 4} in \cite[Theorem~3.2]{DMM1}. 
In this case we have another seiferter $c'$ for $(K, m)$ such that $c$ is a nontrivial cable of $c'$ in $S^3$ and $K$ lies in the interior of $V' = S^3 - \nbhd(c')$. 
Since $c$ is unknotted in $S^3$, 
$c$ wraps $p\ (\ge 2)$ times in the longitudinal direction, 
and wraps exactly once in the meridional direction of $c'$. 
Then $|\lk(K,c) |= p |\lk(K, c')|$ cannot be $1$.  
\end{enumerate}

\end{proof}

\begin{remark}
\label{classification}
There are infinitely many Seifert surgeries $(K, m)$ each of which has a hyperbolic seiferter $c$ with $|\lk(K,c)| = 1$; 
see \cite[Theorem~6.21]{DMM1}. 
\end{remark}

\medskip

Since $S^3 - \nbhd(K_n)$ is the result of $(-\frac{1}{n})$--Dehn filling of $S^3 - \nbhd(K \cup c)$ along $c$, 
Lemma~\ref{link=1seiferter}, 
together with Thurston's Hyperbolic Dehn Surgery Theorem \cite{T1, T2, BePe, PetPorti, BoileauPorti}, 
shows that $K_n$ is a hyperbolic knot for all but finitely many integers $n$. 
Thus there is a constant $N > 0$ such that if $|n| > N$, 
then $K_n$ is a hyperbolic knot in $S^3$ and $(K_n, m_0 + n)$ is a Seifert surgery. 
Hence $K_n(m_0 + n)$ is a Seifert fibered space or a connected sum of lens spaces, 
in particular, it is not a hyperbolic $3$--manifold. 
Now let us recall the following result in \cite{NZ} which follows from the works of Agol \cite{A}, 
Lackenby \cite{La}, and Cao-Meyerhoff \cite{CaMe}.

\medskip

\begin{lemma}[\cite{NZ}]
\label{hyperbolic_genus}
Suppose that $K_n$ is a hyperbolic knot in $S^3$ and 
$K_n(m_0 + n)$ is not a hyperbolic $3$--manifold, 
then $|m_0+n| \le 10.752(2g(K_n) -1)$. 
\end{lemma}

The above inequality shows that when $|n|$ tends to $\infty$, 
the genus $g(K_n)$ of $K_n$ goes to $\infty$. 
Then it follows from Corollary~\ref{lk=1_finite} that 
$\{ (K_n, m_n) \}$ contains only finitely many L-space surgeries. 
Thus the proof of Theorem~\ref{twist_seiferter_genus} is completed. 
\hspace*{\fill} $\square$(Theorem~\ref{twist_seiferter_genus})

\medskip

\section{Braids and L-space knots}
\label{braids}

In this section we investigate Conjectures~\ref{conj:genus} and \ref{conj:genus_twist} from a viewpoint of braids. 
We have observed that for many of the twist families containing infinitely many L-space knots that are studied in \cite{Mote},  
the twisting circle is not only a seiferter but also a braid axis.
Furthermore, 
L-space knots are often isotopic to closures of positive or negative braids.

\subsection{Genera of positive braid closures}
\label{genera_braid_closures}

Well known to the experts,  
we provide a proof of the following.

\begin{proposition}
\label{prop:braid}
There are only finitely many knots of each genus that are closures of positive braids. 
\end{proposition}

\begin{proof}
Observe that there are exactly $(n-1)^\ell$ positive braids in $B_n$ with word length $\ell$. 
Also, if a positive braid $\beta \in B_n$ has word length $\ell$, 
then the oriented closed braid $\widehat{\beta}$ is a fibered link which bounds a Seifert surface (a fiber surface) 
with Euler characteristic $\chi(\widehat{\beta}) = n-\ell$ \cite[Theorem~2]{Sta}. 
Furthermore, if $\ell < 2(n-1)$ then either $\widehat{\beta}$ is a split link of at least two components or there is a positive braid $\beta' \in B_{n-1}$ 
such that $\widehat{\beta}= \widehat{\beta'}$. 
This is because the bound $\ell < 2(n-1)$ implies that some generator of $B_n$ either does not appear in $\beta$ at all or only appears once.  
In the former, the closure necessarily is a split link; in the latter, $\beta$ admits a Markov type of destabilization to $\beta'$.  
See Figure~\ref{fig:positivebraidindex} for an illustration of these two cases.

\begin{figure}[htbp]
\begin{center}
\includegraphics[width=\textwidth]{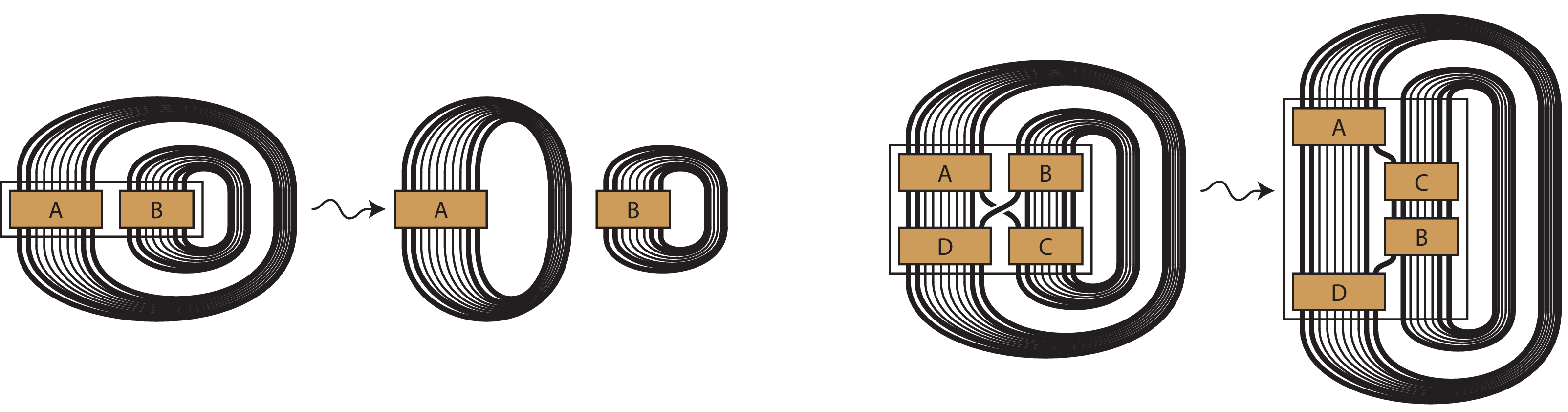}
\caption{If a generator of $B_n$ does not appear in the braid word $\beta$, the link $\widehat{\beta}$ is split.   
If a generator appears just once in the positive braid word $\beta$, 
there is positive braid word $\beta'$ of smaller braid index such that  $\widehat{\beta'}=\widehat{\beta}$. }
\label{fig:positivebraidindex}
\end{center}
\end{figure}

Now assume $K$ is a knot that is the closure of a positive braid.  
Let $n$ be the smallest index such that $K=\widehat{\beta}$ for a positive braid $\beta \in B_n$. 
Since a fibered knot has a unique minimal genus Seifert surface \cite[Lemma 5.1]{EL} (\cite{Thu}),  
$g(K) = \frac{1- \chi(\widehat{\beta})}{2} = \frac{1-n+\ell}{2}$, 
and the word length of $\beta$ is $\ell = 2g(K) + n -1 \geq 2(n-1)$.  
Hence $2g(K)+1 \geq n$.  
Thus the positive braid index of  a knot is bounded above by its genus.  
Therefore for each braid index there are only finitely many positive braids whose closure is a knot of a given genus.
\end{proof}

\medskip

As a direct consequence of Proposition~\ref{prop:braid} we have:  
 
\medskip
 
\begin{corollary}
\label{positive braids}
The class of L-space knots which are closures of positive and negative braids satisfy Conjecture~\ref{conj:genus}.  
\end{corollary}

Note that  the classes of closures of positive braids and L-space knots are distinct.  
For example, KnotInfo shows that the hyperbolic knot $10_{139}$ is the closure of a positive braid \cite{knotinfo}, 
but its Alexander polynomial indicates that it cannot be an L-space knot \cite{OS3}.  
On the other hand, 
the $(3, 2)$--cable of $T_{3, 2}$ is an L-space knot, 
but it is not a closure of any positive or negative braid, e.g.\ \cite[6.3]{HW2}. 

Now let us assume that the twisting circle $c$ is a braid axis for $K$, 
then $K_n$ is the closure of a positive braid for $n \gg 0$ and of a negative braid for $n \ll 0$. 
Hence, even without Theorem~\ref{genera_twist}, 
Proposition~\ref{prop:braid} implies Conjecture~\ref{conj:genus_twist} for the twist family $\{K_n\}$. 

However, a twisting circle $c$ which is not a braid axis for $K$ may provide a twist family $\{ K_n \}$ containing 
infinitely many L-space knots. 
For instance, \cite[Example 1.2]{Mote} 
shows the pretzel knots $K_n=P(-2,3,2n+1)$ with the $7+4n$ surgery is a twist family with $|\lk(K,c)| = 2$ producing Seifert fibered L-spaces for $n\geq0$. 
Since $K$ is in general not a torus knot  while all $2$--braids are, $c$ is not a braid axis.   
Nevertheless, note that these knots are positive $3$--braids when $n \geq0$.  
In the next subsection we provide further such examples.

\subsection{L-space knots obtained by twisting torus knots} 
\label{twisted torus knots}

Recall that torus knots $T_{p, q}$ are fundamental examples of L-space knots.
In the standardly embedded torus $T$ with preferred oriented meridian $m$ and longitude $l$, 
the torus knot $T_{p,q}$ is the unoriented curve in $T$ homologous to $\pm (q [l]+ p [m]) \in H_1(T)$ when given an orientation.  
Since $T_{p, q}$ is unoriented, 
we may choose that $|p| \ge q \ge1$; 
$|p| = q$ if and only if $(p, q) = (\pm 1, 1)$.
Then $T_{p,q}$ is a positive braid and a positive L-space knot whenever 
$p \ge 1$ and a negative braid and a negative L-space knot whenever $p \le -1$.  
(If ever $q=1$, then the knot is actually the unknot, 
and we regard the unknot to be both a positive and a negative braid. 
Recall that only the unknot is a positive and a negative L-space knot.)
We say $T_{p,q}$ is a positive torus knot in the former situation and a negative torus knot in the latter.  
Since $g(T_{p, q}) = \frac{(|p|-1)(q-1)}{2}$, for any given integer $N$, 
there are only finitely many $T_{p, q}$ with $g(T_{p, q}) \le N$.

In the following we show that 
every torus knot $K = T_{p, q}$ has a seiferter $c$ that yields a twist family of Seifert surgeries $(K_n, r_n)$ on hyperbolic knots $K_n$ for all but finitely many $n$ \cite{DMM1}, 
and this twist family contains infinitely many L-space surgeries \cite{Mote}.   
Notably if $p \le -3$ and $2q > |p| > q$, 
then $c$ is not a braid axis. 
However, these L-space knots can be re-arranged as closures of positive or negative braids.

Depicted on the left and right side of Figure~\ref{fig:torusknotseiferters} are unknots $c_+$ and $c_-$ disjoint from a once-punctured $T$, 
though the right side shows $-l$.  
Given the torus knot $T_{p,q}$ in the punctured torus $T$, 
we define $c_\pm$ to be the corresponding knot in the complement of $T_{p,q}$.  
It follows from \cite{DMM1} (where they are called $c^\pm_{p,q}$) that these are seiferters for the torus knots $T_{p,q}$ with the $pq$ surgery.  
The central two images of Figure~\ref{fig:torusknotseiferters} show that $T_{p,q} \cup c_+$ is the mirror of $T_{-p,q} \cup c_-$; 
the mirroring is through a vertical plane containing the curve $m$.  
Hence, by mirroring as needed, 
we may restrict attention to the seiferter $c_+$.

\begin{proposition}
\label{axis}
\
\begin{itemize}
\item 
If $p \ge 1$, or $p \le -2$ and $|p| \ge 2q$, 
then $c_+$ is a braid axis for $T_{p, q}$. 

\item
If $p \le -3$ and $q < |p| < 2q$, 
then $c_+$ is not a braid axis for $T_{p, q}$. 
\end{itemize}
\end{proposition}

\begin{proof}
The first assertion follows from Figures~\ref{fig:positivetorusknotseiferter-posbraid}, \ref{fig:negativetorusknotseiferter-braid} 
and \ref{fig:negativetorusknotseiferter-furtherbraid} (the bottom-left). 
If $p \le -3$ and $q < |p| < 2q$, 
then as shown in Figures~\ref{fig:negativetorusknotseiferter-braid} and \ref{fig:negativetorusknotseiferter-furtherbraid} (the bottom-right), 
$c$ links $T_{p, q}$ coherently, 
meaning that $c$ bounds a disk that $T_{p, q}$ always intersect in the same direction, 
with $| \lk(T_{p, q}, c_+)| = |p| - q < q$.  
By the assumption on $p$ and $q$, 
we have $|p| > q \ge 2$ and the braid index $T_{p, q}$ is known to be $q$ \cite[Proposition~10.5.2]{Crom}, 
and hence $c_+$ cannot be a braid axis for $T_{p, q}$.
\end{proof}

Define $T_{p,q,n}$ to be the result of an $n$--twist of the torus knot $T_{p,q}$ along the seiferter $c_+$.   
If $q = 1$, then $T_{p,1,n}$ is a torus knot $T_{p+1,\, 1 + (p+1)n} = T_{1+(p+1)n,\, p+1}$ for any non-zero integer $p$, 
which is a positive or negative braid and an L-space knot  for all integers $n$.
 
Theorem 1.7 of \cite{Mote} and its proof show that

\begin{itemize}
\item if $p \ge 1$ then $T_{p,q,n}$ is an L-space knot for all integers $n$, and
\item if $p \le -1$ then $T_{p,q,n}$ is an L-space knot for any integer $n \leq 1$.
\end{itemize}

Furthermore, in either case,  
$|\lk(T_{p,q,n}, c_+)| =|p+q|$ and the algebraic linking equals the geometric linking, 
i.e.\ $T_{p,q,n}$ intersects a disk bounded by $c_+$ in the same direction. 
See Figure~\ref{fig:torusknotseiferterasaxis}. 
Thus Theorem~\ref{genera_twist} shows that $g(T_{p, q, n}) \to \infty$ as $|n| \to \infty$ for any $p,q$ with $|p+q| > 1$.  
(If $|p+q|=1$, then $c_+$ is a meridian of $T_{p,q,n}$.  
If $|p+q|=0$, then $-p=q=1$ and $T_{-1,1,n} \cup c_+$ is the unlink, 
in particular, $T_{-1, 1, n}$ is the unknot for all integers $n$.)  
Hence Conjecture~\ref{conj:genus} is satisfied for each twist family of knots $T_{p,q,n}$ individually.  

However, by showing the L-space knots among all the knots $T_{p,q,n}$ are positive or negative braids, 
we may conclude that Conjecture~\ref{conj:genus} is satisfied for the twist families of knots $T_{p,q,n}$ collectively. 

\medskip

\begin{proposition}
\label{prop:twisttorusbraid}
\
\begin{itemize}
\item If $p \ge 1$ then $T_{p,q,n}$ is a positive or negative braid for all integers $n$.
\item If $p \le -1$ then $T_{p,q,n}$ is a positive or negative braid for any integer $n \leq 2$.
\end{itemize}
\end{proposition}

\medskip

\begin{corollary}
\label{Tpqn_conj}
Conjecture~\ref{conj:genus} is satisfied for the collection of knots $T_{p,q,n}$ with either 
$p \ge 1$, $q \ge 1$, 
and all $n$ or $p\le -1$, $q \ge 1$ and $n \leq 2$.
\end{corollary}

\begin{proof}
This follows immediately from Proposition~\ref{prop:braid} and Proposition~\ref{prop:twisttorusbraid}. 
\end{proof}

\medskip

\begin{proof}[Proof of Proposition~\ref{prop:twisttorusbraid}]
We represent the torus knot $T_{p,q}$ in the once punctured torus $T$ by one of two train tracks in $T$ depending on whether $p>0$ or $p<0$ (and requiring $q>0$).    
The knots carried by these train tracks after $(-\frac{1}{n})$--surgery on $c_+$ are our knots $T_{p,q,n}$.   
By a sequence of isotopies of $T$, $c_+$, 
and the train tracks along with splittings of the train tracks we will arrange the train tracks into positions where it is apparent that they carry positive or negative braids after $(-\frac{1}{n})$--surgery on $c_+$ for particular values of $n$. 

For $p \ge 1$, 
Figure~\ref{fig:positivetorusknotseiferter-posbraid} shows that the knot $T_{p,q,n}$ is actually a positive braid if $n \geq 0$ and a negative braid if $n \leq -1$.   

Suppose that $p \le -1$. 
If $p = -1$, 
then $q = 1$ and $T_{-1, 1, n}$ is the unknot for all integers $n$. 
For $p \le -2$, 
Figure~\ref{fig:negativetorusknotseiferter-braid} shows that the knot $T_{p,q,n}$ is a negative braid if $n \leq 0$. 
Continuing from this, 
Figure~\ref{fig:negativetorusknotseiferter-furtherbraid} indicates how to further isotope $T_{p,q,n}$ 
(with $p \le -2$) into a positive or negative braid for $n=1$ or $n=2$.    

First assume $n=2$.  
If $2q > |p|$ (as on the right side of Figure~\ref{fig:negativetorusknotseiferter-furtherbraid}), 
then the knot may be isotoped to a negative braid.   
If $|p| > 2q$ (as on the left side of Figure~\ref{fig:negativetorusknotseiferter-furtherbraid}), 
then the knot may be isotoped to a positive braid.  

Now assume $n=1$.  
Then we may discard the twisting circle in Figure~\ref{fig:negativetorusknotseiferter-furtherbraid}.    
If $2q> |p|$ (as on the right side of Figure~\ref{fig:negativetorusknotseiferter-furtherbraid}), 
then the knot may be isotoped to a negative braid.   
If $|p| > 2q$ (as on the left side of Figure~\ref{fig:negativetorusknotseiferter-furtherbraid}), 
then the knot may be isotoped into a configuration similar to the initial configuration at the top of Figure~\ref{fig:negativetorusknotseiferter-furtherbraid}, 
but with smaller braid index and mirrored (say, mirrored across a horizontal line below the diagram). 
This argument can now be repeated until a positive or negative braid, 
or an index-one braid (i.e. the unknot) is achieved.
\end{proof}

\medskip

\begin{question}
Which  knots $T_{p,q,n}$ with $p \le -2$, $q \ge 2$, and $n \geq 2$ are L-space knots?  
\end{question}

\begin{figure}
\begin{center}
\includegraphics[width=\textwidth]{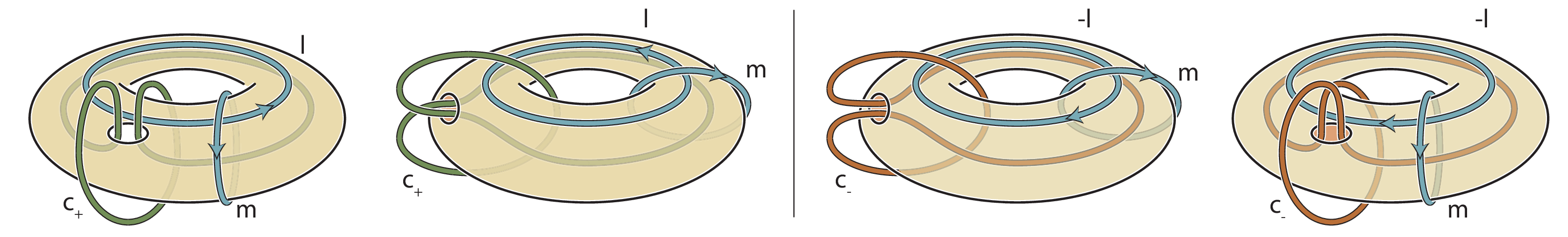}
\caption{The seiferter $c_+$ for a $(p,q)$--torus knot is mirror equivalent to the seiferter $c_-$ for a $(-p,q)$--torus knot.}
\label{fig:torusknotseiferters}
\end{center}
\end{figure}

\begin{figure}
\begin{center}
\includegraphics[width=\textwidth]{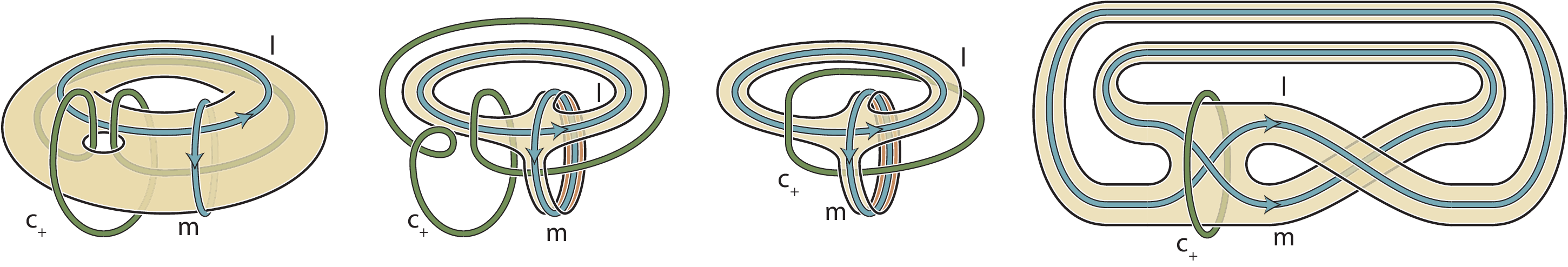}
\caption{The standardly embedded once-punctured torus, its preferred meridian-longitude basis, 
and the seiferter $c_+$ for the torus knots carried by this torus are isotoped into a convenient configuration.}
\label{fig:torusknotseiferterasaxis}
\end{center}
\end{figure}

\begin{figure}
\begin{center}
\includegraphics[width=\textwidth]{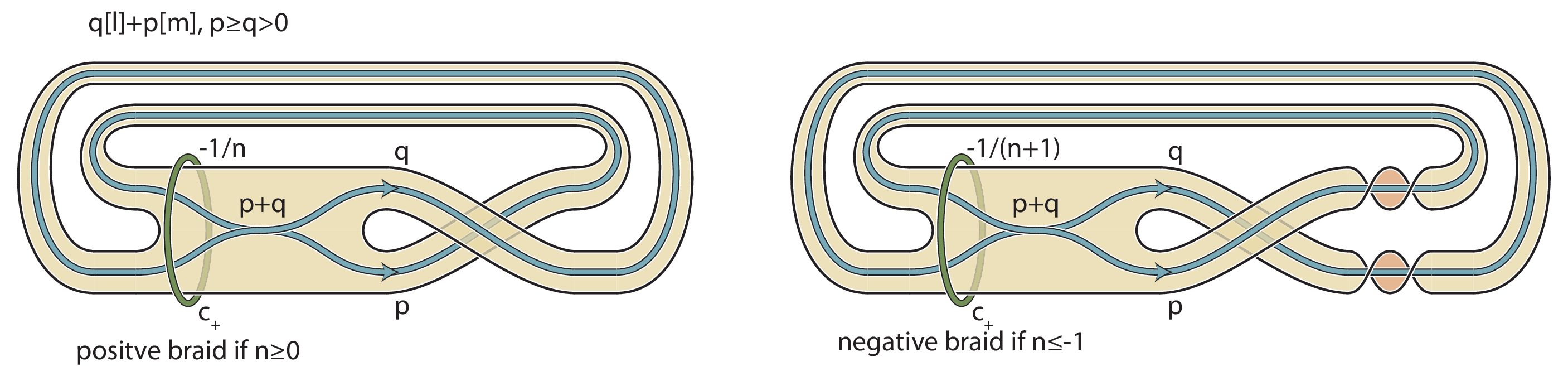}
\caption{Beginning with a positive torus knot, 
$(-\frac{1}{n})$--surgery on the seiferter $c_+$ produces a closed positive braid if $n \geq 0$ and a closed negative braid if $n\leq -1$.}
\label{fig:positivetorusknotseiferter-posbraid}
\end{center}
\end{figure}

\begin{figure}
\begin{center}
\includegraphics[width=\textwidth]{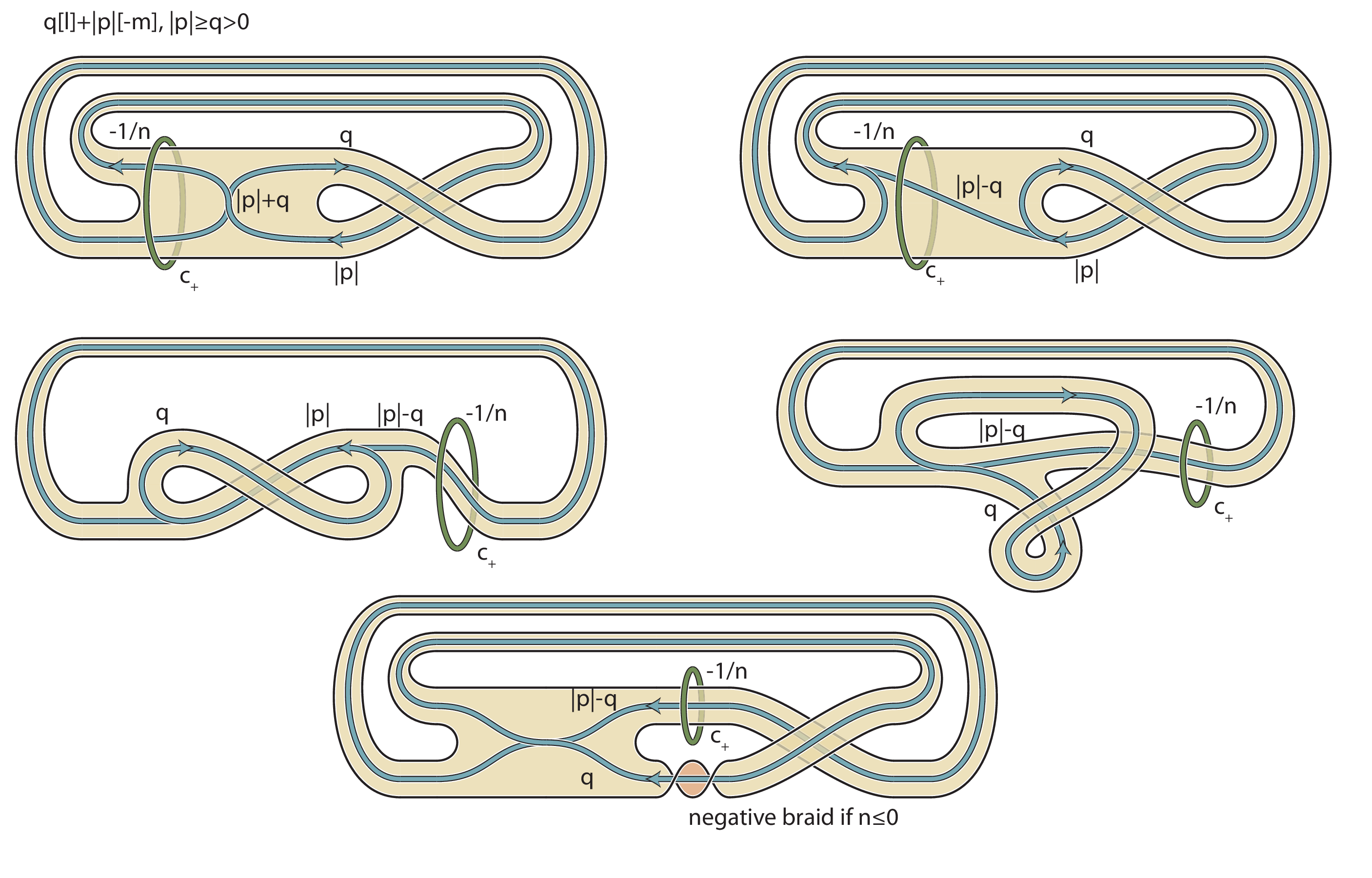}
\caption{Beginning with a negative torus knot, 
$(-\frac{1}{n})$--surgery on the seiferter $c_+$ produces a closed negative braid if $n \leq 0$.}
\label{fig:negativetorusknotseiferter-braid}
\end{center}
\end{figure}

\begin{figure}
\begin{center}
\includegraphics[width=\textwidth]{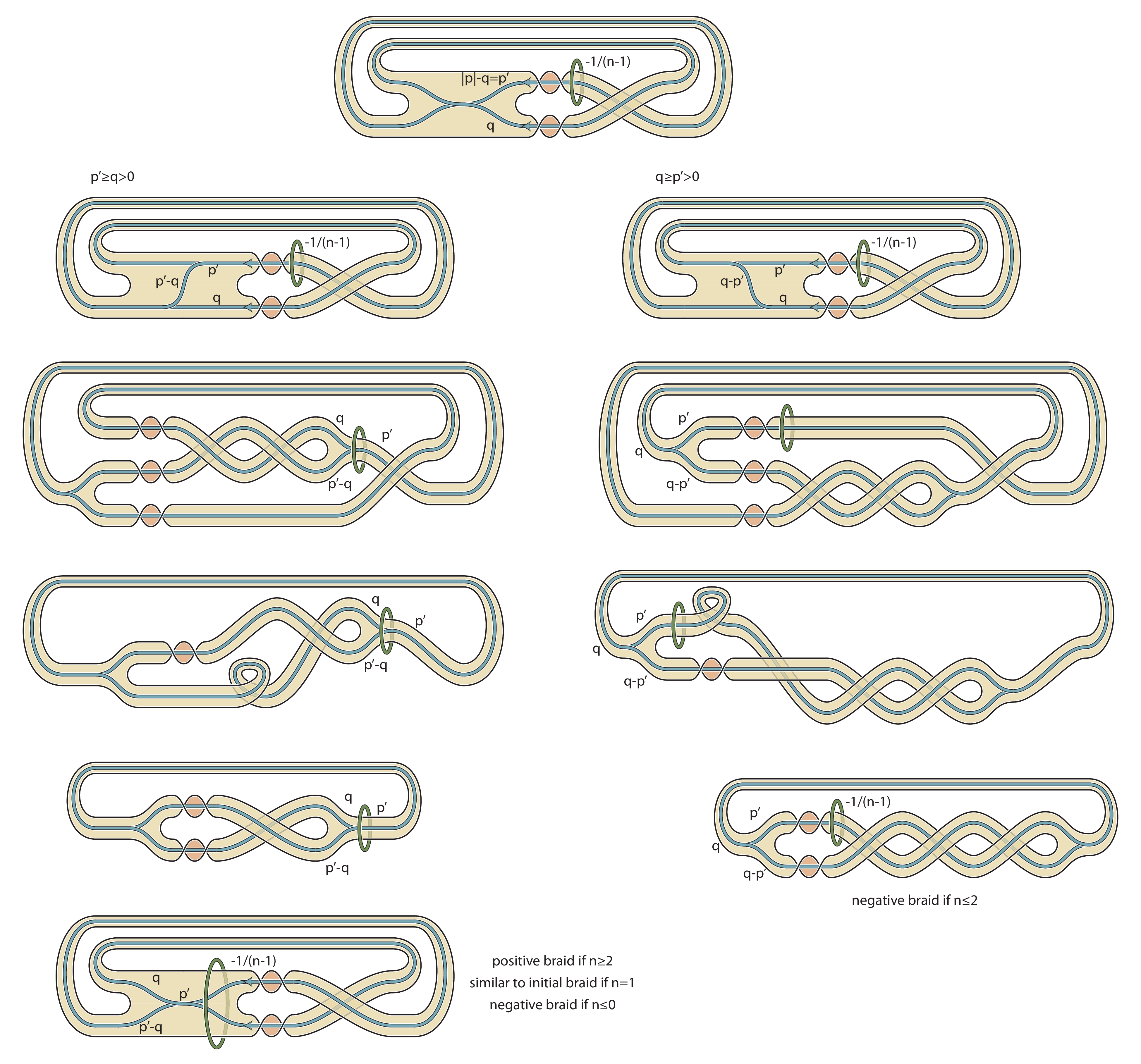}
\caption{Continuing from the end of Figure~\ref{fig:negativetorusknotseiferter-braid}, two possibilities are examined.  
Down the left, when $p' = |p|-q > q$: 
either $n\neq1$ and the knot can be isotoped to a positive or negative braid; 
or $n=1$ and the knot can be rearranged into a mirrored form of the initial position but with smaller braid index.  
Down the right, when $q > p'$, the knot can be isotoped into a negative braid  if $n \leq 2$. }
\label{fig:negativetorusknotseiferter-furtherbraid}
\end{center}
\end{figure}

\medskip

\section{Questions}
\label{questions}

We close this article with a few questions which arise in our study. 

\begin{question}
Is there a twist family $\{K_n\}$ containing infinitely many hyperbolic L-space knots that are not closures of positive or negative braids?
\end{question}

While no example in \cite{Mote} appears to give a positive answer to this question, we still expect such a twist family to exist.

Although we don't expect twist families containing infinitely many L-space knots to generically be closures of positive or negative braids, 
it still seems plausible that the knots should wrap coherently about the twisting circle.  

\medskip

\begin{question} 
\label{alg=geom}
If a twist family of knots $\{ K_n \}$ obtained by twisting $K$ about an unknot $c$ contains infinitely many L-space knots, 
then does $c$ bound a disk that $K$ always intersects in the same direction?
\end{question}
This behavior is observed in the examples of \cite{Mote}.  
Furthermore, in the case that $|\lk(K,c)|=1$, 
a positive answer would imply that $c$ is a meridian of $K$. 

\bigskip

\textbf{Acknowledgements} --
We would like to thank the referees for their careful readings and valuable suggestions that have enabled us to improve the article.

The first named author was partially supported by a grant from the Simons Foundation (\#209184 to Kenneth L.\ Baker).  
He would like to thank the Mathematics Department of Boston College and CIRGET at Universit\'{e} du Qu\'{e}bec \`{a} Montr\'{e}al for their hospitality during portions of the writing.
The second named author was partially supported by JSPS KAKENHI Grant Number JP26400099 and Joint Research Grant of Institute of Natural Sciences at Nihon University for 2015.

\address{Department of Mathematics\\ University of Miami\\
Coral Gables, FL 33146\\ USA}
\email{k.baker@math.miami.edu}

\address{Department of Mathematics\\ Nihon University\\ 
3-25-40 Sakurajosui\\ Setagaya-ku\\
Tokyo 156--8550\\ Japan}
\email{motegi@math.chs.nihon-u.ac.jp}

\bigskip

\begin{thebibliography}{99}

\bibitem{A}
I. Agol;
Bounds on exceptional Dehn filling,
Geom.\ Topol.\ \textbf{4} (2000), 431--449.  

\bibitem{BePe}
R. Benedetti and C. Petronio; 
Lectures on hyperbolic geometry, 
Universitext, Springer-Verlag, 1992. 

\bibitem{Berge1} 
J. Berge; 
The knots in $D^2 \times S^1$ which have nontrivial Dehn surgeries 
that yield $D^2 \times S^1$, 
Topology Appl.\ \textbf{38} (1991), 1--19.

\bibitem{BL}
S.A. Bleiler and R.A. Litherland; 
Lens spaces and Dehn surgery, 
Proc.\ Amer.\ Math.\ Soc.\ \textbf{107} (1989), 1127--1131.  

\bibitem{BoileauPorti}
M. Boileau and J. Porti; 
Geometrization of $3$--orbifolds of cyclic type,  
Ast\'erisque \textbf{272} (2001), 208pp. 

\bibitem{BuZ}
G. Burde and H. Zieschang; 
Neuwirthsche Knoten und Fl\"achenabbildungen, 
Abh.\ Math.\ Sem.\ Univ.\ Hamburg \textbf{31} (1967), 239--246. 

\bibitem{CaMe}
C. Cao and R. Meyerhoff; 
The orientable cusped hyperbolic $3$--manifolds of minimum volume, 
Invent.\ Math.\ \textbf{146} (2001), no.3, 451--478.

\bibitem{knotinfo} 
J. C. Cha and C. Livingston;
 KnotInfo: Table of Knot Invariants,
 {\verb!http://www.indiana.edu/~knotinfo!},  May 13, 2015.

\bibitem{linkinfo}
J. C. Cha and C. Livingston; 
LinkInfo: Table of Link Invariants, 
{\verb!http://www.indiana.edu/~linkinfo!}, September 19, 2014

\bibitem{Crom}
P. Cromwell; 
Knots and links, 
Cambridge University Press, Cambridge, 
2004. 

\bibitem{CGLS}
M. Culler, C.McA. Gordon, J. Luecke and P.B. Shalen; 
Dehn surgery on knots, 
Ann.\ Math.\ \textbf{125} (1987), 237--300. 

\bibitem{DMM1} 
A. Deruelle, K. Miyazaki and K. Motegi; 
Networking Seifert surgeries on knots, 
Mem.\ Amer.\ Math.\ Soc.\ \textbf{217} (2012), no. 1021, viii+130. 

\bibitem{EL} 
A. Edmonds and C. Livingston; 
Group actions on fibered three-manifolds, 
Comment.\ Math.\ Helv.\ \textbf{58} (1983), 529--542. 

\bibitem{GabaiII}
D. Gabai; 
Foliations and the topology of $3$--manifolds. II, 
J.\ Diff.\ Geom.\ \textbf{26} (1987), 461--478. 

\bibitem{GabaiIII}
D. Gabai; 
Foliations and the topology of $3$--manifolds. III, 
J.\ Diff.\ Geom.\ \textbf{26} (1987), 479--536. 

\bibitem{Gabai_solidtorus}
D. Gabai; 
Surgery on knots in solid tori, 
Topology \textbf{28} (1989), 1--6. 

\bibitem{Ghi}
P. Ghiggini;
Knot Floer homology detects genus-one fibred knots,
Amer.\ J.\ Math.\ \textbf{130} (2008), 1151--1169.

\bibitem{GS}
R. Gompf and A. Stipsicz; 
$4$--Manifolds and Kirby Calculus, 
Graduate Studies in Mathematics Vol.20,  
American Mathematical Society, Providence, 1999. 

\bibitem{GonAcu}
F. Gonz\'alez-Acu\~na; 
Dehn’s construction on knots, 
Bol.\ Soc.\ Mat.\ Mexicana \textbf{15} (1970), 
58--79. 


\bibitem{Go_satellite}
C.McA. Gordon; 
Dehn surgery and satellite knots, 
Trans.\ Amer.\ Math.\ Soc.\  \textbf{275} (1983), 687--708. 

\bibitem{Go_Boundary}
C.McA. Gordon; 
Boundary slopes of punctured tori in $3$--manifolds, 
Trans.\ Amer.\ Math.\ Soc.\  \textbf{350} (1998), 1713--1790. 

\bibitem{Go_small}
C.McA. Gordon; 
Small surfaces and Dehn filling, 
Geometry \& Topology Monographs, 
Vol. 2, Proceedings Kirbyfest, 1999, 177--199. 

\bibitem{GLu}
C.McA. Gordon and J. Luecke; 
Only integral Dehn surgery can yield reducible manifolds, 
Math.\ Proc.\ Camb.\ Phil.\ Soc.\ \textbf{102} (1987), 97--101. 

\bibitem{GLreducible}
C.McA. Gordon and J. Luecke;
Reducible manifolds and Dehn surgery, 
Topology \textbf{35} (1996), 385--409.

\bibitem{GL_TS}
C.McA. Gordon and J. Luecke; 
Toroidal and boundary-reducing Dehn fillings, 
Topology Appl.,\ \textbf{93} (1999), 77--90. 

\bibitem{GWuAnnular}
C.McA. Gordon and Y-Q. Wu;
Annular Dehn fillings,
Comment.\ Math.\ Helv.\ \textbf{75} (2000), 430--456.

\bibitem{GW_AS}
C.McA. Gordon and Y-Q. Wu;
Annular and boundary reducing Dehn fillings,
Topology\ \textbf{39} (2000), 531--548. 

\bibitem{Greene}
J. E. Greene; 
L-space surgeries, genus bounds, and the cabling conjecture, 
J.\ Diff.\ Geom.\ \textbf{100} (2015), 491--506. 

\bibitem{Hat2}
A.E. Hatcher; 
Notes on basic $3$--manifold topology (2000).
Freely available at {\verb!http://www.math.cornell.edu/~hatcher!}

\bibitem{Hed_positive}
M. Hedden;
Notions of positivity and the Ozsv\'ath-Szab\'o concordance invariant, 
J.\ Knot Theory Ramifications \textbf{19} (2010), 617--629.

\bibitem{HW2}
M. Hedden and L. Watson; 
On the geography and botany of knot Floer homology, 
2014, \texttt{arXiv:math/1404.6913}. 


\bibitem{Ja}
W. Jaco;
Lectures on three manifold topology, 
CBMS Regional Conference Series in Math., vol. 43,
Amer.\ Math.\ Soc., 1980.

\bibitem{JS} 
W. Jaco and P. Shalen; 
Seifert fibered spaces in $3$--manifolds, 
Mem.\ Amer.\ Math.\ Soc.\ \textbf{21}  (1979),\ no. 220, viii+192. 

\bibitem{Jo} 
K. Johannson; 
Homotopy equivalences of $3$--manifolds with boundaries, 
Lect.\ Notes in Math.\ vol.\ 761,\ Springer-Verlag,\ 1979.  

\bibitem{Juh}
A. Juh\'asz; 
Floer homology and surface decompositions, 
Geom.\ Topol.\ \textbf{12} (2008), 299--350.

\bibitem{knotatlas} 
{The {K}not {A}tlas};
\texttt{http://katlas.org}

\bibitem{KMS}
M. Kouno, K. Motegi and T. Shibuya;   
Twisting and knot types, 
J.\ Math.\ Soc.\ Japan \textbf{44} (1992), 199-216. 

\bibitem{La}
M. Lackenby;
Word hyperbolic Dehn surgery,
Invent.\ Math.\ \textbf{140} (2000), no.2, 243--282. 

\bibitem{LM}
M. Lackenby, R. Meyerhoff;
The maximal number of exceptional Dehn surgeries,
Invent.\ Math.\ \textbf{191} (2013), no. 2, 341--382. 

\bibitem{LS}
P. Lisca and A. Stipsicz;
Ozsv\'ath-Szab\'o invariants and tight contact 3-manifolds, III,  
J.\ Symplectic Geom.\ \textbf{5}, (2007), 357--384. 

\bibitem{Mano}
C. Manolescu; 
An introduction to knot Floer homology, 
to appear in Proceedings of the 2013 SMS Summer School on Homology Theories of Knots and Links

\bibitem{Ma} 
Y. Mathieu; 
Unknotting, knotting by twists on disks and property (P) for 
knots in $S^3$, Knots 90 (ed. by Kawauchi), 
Proc. 1990 Osaka Conf. on Knot Theory and Related Topics, 
de Gruyter (1992), 93-102

\bibitem{McCoy}
D. McCoy; 
Surgeries, sharp $4$--manifolds and the Alexander polynomial, 
2014, \texttt{arXiv:math/1412.0572}. 

\bibitem{MM3}
K. Miyazaki and K. Motegi; 
Seifert fibered manifolds and Dehn surgery III,
Comm.\ Anal.\ Geom.\ \textbf{7} (1999), 551--582. 

\bibitem{Mote}
K. Motegi; 
L-space surgery and twisting operation, 
Algebr.\ Geom.\ Topol.\ \textbf{16} (2016), 1727--1772.  

\bibitem{Ni}
Y. Ni; 
Knot Floer homology detects fibred knots, 
Invent.\ Math.\  \textbf{170} (2007) 577--608. 

\bibitem{Ni2}
Y. Ni; 
Erratum: Knot Floer homology detects fibred knots, 
Invent.\ Math.\ \textbf{177} (2009), 235--238. 

\bibitem{Ni_fibred}
Y. Ni; 
Dehn surgeries that yield fibred $3$--manifolds, 
Math.\ Ann.\ \textbf{344} (2009), 863--876. 

\bibitem{Ni3}
Y. Ni; 
Dehn surgeries on knots in product manifolds, 
J.\ Topology \textbf{4} (2011), 799--816. 

\bibitem{NZ}
Y. Ni and X. Zhang; 
Characterizing slopes for torus knots, 
Algebr.\ Geom.\ Topol.\ \textbf{14} (2014), 1249--1274. 

\bibitem{Oh}
S. Oh; 
Reducible and toroidal manifolds obtained by Dehn filling, 
Topology Appl. \textbf{75} (1997), 93--104.

\bibitem{OS_genus}
P. Ozsv\'ath and Z. Szab\'o; 
Holomorphic disks and genus bounds, 
Geom.\ Topol.\ \textbf{8} (2004), 311--334. 

\bibitem{OS_knot}
P. Ozsv\'ath and Z. Szab\'o; 
Holomorphic disks and knot invariants, 
Adv.\ Math.\ \textbf{186} (2004), 58--116. 

\bibitem{OS3} 
P. Ozsv\'ath and Z. Szab\'o; 
On knot Floer homology and lens space surgeries, 
Topology \textbf{44} (2005), 1281--1300. 

\bibitem{OS4} 
P. Ozsv\'ath and Z. Szab\'o; 
Knot Floer homology and rational surgeries, 
Algebr.\ Geom.\ Topol.\ \textbf{11} (2011), 1--68. 

\bibitem{PetPorti}
C. Petronio and J. Porti;
Negatively oriented ideal triangulations and a proof of Thurston's hyperbolic Dehn filling theorem, 
Expo.\ Math.\ \textbf{18} (2000), 1--35. 

\bibitem{Q}
R. Qiu; 
Reducible Dehn surgery and toroidal Dehn surgery, 
Northeastern Mathematical Journal \textbf{13} (1997), no. 4, 453--458.

\bibitem{Ras}
J.A. Rasmussen; 
Floer homology and knot complements, 
Ph.D. thesis, Harvard University, 2003, 
\texttt{arXiv:math/0306378}. 

\bibitem{Ro}
D. Rolfsen; 
Knots and links, 
Publish or Perish, Berkeley, Calif., 1976.  

\bibitem{Sch}
M. Scharlemann; 
Producing reducible $3$--manifolds by surgery on a knot, 
Topology \textbf{29} (1990), 481--500. 

\bibitem{ST}
M. Scharlemann and A. Thompson; 
Surgery on a knot in $\mathrm{Surface} \times I$, 
Algebr.\ Geom.\ Topol.\ \textbf{9} (2009), 1825--1835. 

\bibitem{Sta}
J.R. Stallings; 
Constructions of fibered knots and links, 
Proc.\ Symposia in Pure Mathematics, 
vol. 32, 1978, 55--60. 

\bibitem{T1}
W.P. Thurston; 
The geometry and topology of $3$--manifolds, 
Lecture notes, Princeton University, 1979. 

\bibitem{T2}
W.P. Thurston; 
Three dimensional manifolds, Kleinian groups and hyperbolic geometry, 
Bull.\ Amer.\ Math.\ Soc.\ \textbf{6} (1982), 357--381. 

\bibitem{Thu}
W.P. Thurston; 
A norm for the homology of $3$--manifolds,  
Mem.\ Amer.\ Math.\ Soc.\
\textbf{59} (1986), no. 339, i--vi, 99--130. 

\bibitem{Torres}
G. Torres;
On the Alexander polynomial, 
Ann.\ of Math.\ (2) \textbf{57}, (1953). 57--89.

\bibitem{WW}
S. Wang and Y-Q. Wu; 
Covering invariants and cohopficity of $3$--manifold groups, 
Proc.\ London.\ Math.\ Soc.\ \textbf{68} (1994), 203--224. 

\bibitem{Wu_incomp}
Y-Q. Wu; 
Incompressibility of surfaces in surgered $3$--manifolds, 
Topology\ \textbf{31} (1992), 271--279. 

\bibitem{Wu_suture}
Y-Q. Wu; 
Sutured manifold hierarchies, essential laminations, and Dehn surgery, 
J.\ Diff.\ Geom.\ \textbf{48} (1998), 407--437. 

\bibitem{Y}
R. Yamamoto; 
Open books supporting overtwisted contact structures and the Stallings twist, 
J.\ Math.\ Soc.\ Japan \textbf{59} (2007), 751--761. 

\end{thebibliography}
\end{document}